\documentclass[12pt]{article}
\usepackage{amsmath,amsthm,verbatim,amssymb,amsfonts,amscd,diagrams, graphicx, mathrsfs, mathtools}
\usepackage[usenames,dvipsnames]{color}
\usepackage{amstext}
\usepackage{color}
\usepackage{hyperref}

\theoremstyle{plain}
\newtheorem{theorem}{Theorem}[section]
\newtheorem{corollary}[theorem]{Corollary}
\newtheorem{lemma}[theorem]{Lemma}

\newtheorem{proposition}[theorem]{Proposition}

\theoremstyle{definition}
\newtheorem{definition}[theorem]{Definition}

\theoremstyle{remark}
\newtheorem{remark}[theorem]{Remark}

\newcommand{\codim}{\textup{codim}}
\newarrow{ul}---->
\newarrow{Backwards}<----

\newcommand{\xr}{\xrightarrow}

\newcommand{\Z}{\mathbb{Z}}

\newcommand{\R}{\mathbb{R}}

\newcommand{\D}{\mathbb{D}}

\newcommand{\bd}{\partial}
\newcommand{\pf}{\pitchfork}

\newcommand{\mc}[1]{\mathcal{#1}}

\newcommand{\dlim}{\varinjlim}

\newcommand{\mf}{\mathfrak}

\newarrow{Onto}----{>>}
\newarrow{Equals}=====

\begin{document}
\title{The chain-level intersection product for PL pseudomanifolds revisited}
\author{Greg Friedman\thanks{
This  work partially supported by the National Science Foundation under Grant Number (DMS-1308306) to Greg Friedman. }\\
Texas Christian University\\
Fort Worth, TX\\
greg.b.friedman@gmail.com
}

\date{January 2, 2018}

\maketitle

\begin{abstract}

We generalize the PL intersection product for chains on PL manifolds and for intersection chains on PL stratified pseudomanifolds
to products  of locally finite chains on non-compact spaces that are natural with respect to restriction to open sets. This is necessary to sheafify the intersection product, an essential step in proving duality between the Goresky-MacPherson intersection homology product and the intersection cohomology cup product pairing recently defined by the author and McClure. 

We also provide a correction to the Goresky-MacPherson proof of a version of Poincar\'e duality on pseudomanifolds that is used in the construction of the intersection product. 

\end{abstract}

\medskip




\section{Introduction}
The purpose of this paper it to construct an intersection product for not-necessarily compact PL intersection chains on not-necessarily compact PL stratified pseudomanifolds that is natural with respect to restriction to open sets\footnote{\emph{Intersection chains} are the chains whose homology gives the Goresky-MacPherson \emph{intersection homology} theory; see Section  \ref{S: IC} for details. The reader should take care to distinguish between intersection chains, intersection \emph{of} chains, and intersection of intersection chains.}. Such a construction is required to sheafify the intersection product, which will be used in \cite{GBF30} to prove duality between the intersection cohomology cup product and the Goresky-MacPherson intersection homology intersection product. To explain more fully, we begin with some background. 

\paragraph{Intersection products.}
The investigation of duality pairings on manifolds using intersection of geometric objects in general position dates back to Poincar\'e \cite{Poin1895} and Lefschetz \cite{Lef, LefBook}. While this point of view was largely supplanted by the development of cohomology and the cup product, intersection products nonetheless remained relevant. As Dold observes, they are ``closer to geometric intuition and therefore possess considerable heuristic value; they often indicate how to turn an intuitive geometric result $\ldots$ into a rigorous one. $\ldots$ Intersection-products can also serve to \emph{compute} $\smile$-products in manifolds'' \cite[Section VIII.13.26]{Dold}.

The importance of intersection products was elevated considerably in the 1980s when Goresky and MacPherson  \cite{GM1}
extended Poincar\'e duality to piecewise linear (PL) stratified pseudomanifolds (Definition \ref{D: pseudomanifold}), a class of spaces more general than PL manifolds and including, for example, irreducible singular varieties. They did this by constructing a new  intersection product for a version of chains called \emph{intersection chains}. The resulting homology theory is called \emph{intersection homology}, and Poincar\'e duality on pseudomanifolds was first expressed in terms of nonsingular pairings on certain intersection homology groups. When the space is in fact a manifold, the Goresky-MacPherson construction provided a new chain-level definition for the classical intersection product.
By using the cup product as part of their definition of the intersection product and by working with PL chains, they avoided many of the complexities and  technical difficulties in Lefschetz's constructions (cf.\ \cite[Section 1.II.4.D]{Dieudonne}). For example, PL chains live in the direct limit of the usual simplicial chain complexes under subdivision, and consequently they are determined by the PL structure of the space but not any specific triangulation choices, making them more natural geometric objects and simplifying general position requirements.

On pseudomanifolds, the Goresky-MacPherson intersection product is defined on certain pairs of PL chains
in \emph{stratified general position}, meaning that general position holds within each manifold stratum of the space. 
So the the product is not given by a chain map, but it nonetheless induces pairings on homology. On a compact oriented PL $n$-manifold $M$, for example, we obtain products
\begin{equation}\label{E: GM hom}
H_i(M)\otimes H_j(M)\to H_{i+j-n}(M)
\end{equation}
 by intersecting representative cycles in general position \cite[Section 2]{GM1}, and similarly one can generalize to products  of certain intersection homology groups on pseudomanifolds. In the manifold case, the product \eqref{E: GM hom} is Poincar\'e dual to the cup product in cohomology, as we shall see (Corollary \ref{C: all agree}). 

\paragraph{Relationships with other products.}
This duality between the Goresky-MacPherson intersection product and the cup product on compact manifolds is not surprising, as it is well known that Poincar\'e duality on manifolds manifests in many forms. Depending on the setting, this can include nonsingular pairings in homology or cohomology induced by constructions on chains (intersection products), cochains (cup products), differential forms (wedge products), or complexes of sheaves. Versions of all of these pairings also now exist in the setting of pseudomanifolds utilizing intersection homology and cohomology \cite{GM1, GM2, BHS, GBF25}. However, these various pairings are not \emph{a priori} isomorphic in either context, and so it is necessary to develop methods of comparison.
Such compatibilities are often taken for granted, but they are not always straightforward, especially those involving intersection products. Even in the manifold case, it can be hard to come by detailed proofs of the isomorphism between the cohomology cup product and the homology intersection product as  induced by a chain-level pairing, such as Goresky and MacPherson's. For example, 
in the 1970s Dold \cite[Section VIII.13]{Dold} gave a product on manifolds, defined only at the level of homology, that is essentially built to be Poincar\'e dual to the cup product, but it takes some work to see that this product is induced by some chain-level construction. In fact, it turns out to be the same homology product as that determined by the Goresky-MacPherson chain-level product \eqref{E: GM hom} (see Corollary \ref{C: all agree}).

So  the relationship between cup and intersection products on compact manifolds really does end up as expected, and this can be shown using only the tools of classical algebraic topology (see Section \ref{S: cup dual}). By contrast, the case of intersection homology and cohomology on pseudomanifolds proves less straightforward. Following the original work of Goresky and MacPherson on intersection products \cite{GM1}, many of the standard tools of algebraic topology have since been extended to intersection homology and cohomology. Very recently this includes cup and cap products and a Poincar\'e duality theorem given by cap product with a fundamental class \cite{GBF25, GBF35}. However, 
as we will see in Section \ref{S: open}, there does not appear to be a simple way to demonstrate duality between  the cup and intersection products by direct classical means (e.g.\ without sheaf theory). Yet the existence of such an isomorphism has important applications. For example, while cup products are useful theoretical tools, the intersection cohomology cup product does not seem to admit combinatorial computation as does the ordinary simplicial cup product \cite[Introduction to Chapter 7]{GBF35}, and so the intersection product remains a key computational tool in intersection homology in the spirit of Dold's quote above. For a recent interesting example see \cite{Sch15a,Sch15b}. 

Without a proof by classical techniques, it is necessary to utilize other tools to prove that the intersection homology intersection product is compatible with other products on pseudomanifolds.  
The most convenient \emph{lingua franca} for such comparisons seems to be the derived category of sheaves, which possesses powerful axiomatic tools for comparing maps \cite[Section V.9]{Bo}, and we pursue this approach in \cite{GBF30}. Products  to be considered in \cite{GBF30}  include the intersection product discussed here, the Goresky-MacPherson sheaf product of \cite{GM2}, the  cup product of \cite{GBF25}, and the wedge product of intersection differential forms of Brasselet-Hector-Saralegi \cite{BHS} and Saralegi \cite{S1, Sa05}. On manifolds these products reduce to the usual  homology/cohomology pairings, and so we also recover many classical equivalences as special cases.

\paragraph{Sheafification of the intersection product.}
To carry out the program of \cite{GBF30}, it is necessary to have a version of the chain-level intersection product that can be ``sheafified.'' 
While sheafifying cup products of cochains and wedge products of differential forms is straightforward due to their contravariant functoriality, sheafifying the intersection product is not. In particular, it requires a version of the intersection product given by chain maps and behaving contravariantly with respect to the inclusion of open subsets. We thus need a chain-level intersection product with the following properties\footnote{Another well-known way to sheafify complexes of chains on a space $X$, including complexes of singular chains, is via functors  $U\to C_*(X,X-\bar U)$ for open $U\subset X$. However, in the PL setting the sheafification discussed here is more standard \cite{GM2,Bo} and has several useful sheaf theoretic properties; in particular we obtain a complex of \emph{soft} sheaves \cite[Section II.5]{Bo}.}:

\begin{enumerate}
\item The product should be defined on \emph{not necessarily compact} PL stratified pseudomanifolds.

\item The product should accept as inputs \emph{not necessarily compact} PL chains and PL intersection chains
 (sometimes called locally finite or Borel-Moore PL chains) in appropriate general position.

\item The product should be natural with respect to restriction to open subsets. For example, if $U\subset X$ is an open subset,  $\xi$ and $\eta$ are appropriate chains, and $\pf$ denotes the intersection product, then $(\xi\pf\eta)|_U=\xi|_U\pf\eta|_U$ in the appropriate chain complex on $U$. 

\item The product should be formulated in terms of chains maps, not just on pairs of chains in general position.
\end{enumerate}

We construct such a product, and our main conclusion is the following theorem. Perversities $\bar p$ and intersection chains $I^{\bar p}C^\infty_*(X)$ will be reviewed in Section \ref{S: IC}. The $\infty$ decorations denote locally finite chains, and the maps $\mu$ are induced by intersections of chains. Other expressions in the theorem will be explained elsewhere below. The theorem statement itself follows directly by assembling Definitions \ref{D: int pairing} and \ref{D: GP} and Propositions \ref{P: qi}, \ref{P: res}, \ref{P: iqi}, \ref{P: G}, \ref{P: ires}, and \ref{P: compact}.

\begin{theorem}\label{T: main}
Let $X$ be an oriented PL stratified pseudomanifold of dimension $n$ with singular locus $\Sigma$, and let $P=(\bar p_1,\bar p_2)$ be a pair of perversities on $X$. Then there exist well-defined  chain maps
\begin{align}
C_*^\infty(X,\Sigma)\otimes C_*^\infty(X,\Sigma)&\hookleftarrow\mf G^\infty_*(X,\Sigma)\xr{\mu} C_{*-n}^\infty(X,\Sigma)\label{E1}\\
I^{\bar p_1}C_*^\infty(X)\otimes I^{\bar p_2}C_*^\infty(X)&\hookleftarrow G^{\infty,P}_*(X)\xr{\mu}I^{\bar p_1+\bar p_2}C_{*-n}^\infty(X)\label{E2}
\end{align}
such that

\begin{enumerate}
\item the leftward inclusion maps are quasi-isomorphisms, i.e.\ they induce isomorphisms of all homology groups,
\item all maps are natural with respect to restrictions to open subsets,
\item the maps in \eqref{E2} are restrictions of those in \eqref{E1} to subcomplexes,  and
\item the maps \eqref{E2} induce the Goresky-MacPherson intersection homology product on compact $X$. 
\end{enumerate}
\end{theorem}

Limiting the discussion momentarily to compact $n$-manifolds for simplicity, the reader may be surprised that our intersection products are not given by chain maps of the form $C_*(M)\otimes C_*(M)\to C_{*-n}(M)$. However, general position requirements rule out intersection maps of this form because, for example, we can't generally intersect a chain with itself. In fact, there is no hope of producing a chain map of this form (over $\Z$) that is graded commutative (as is the intersection pairing) and produces a homology product isomorphic to the cup product. This is a consequence of the failure of the ``commutative cochain problem'' (see \cite[Section IX.A]{GrMo}). Rather, as an alternative way to construct a homology intersection product  using chain maps, McClure \cite{McC} defined the \emph{domain}  for an arity-$k$ intersection product to be a subcomplex $G^k_*(M)\subset C_*(M)^{\otimes k}$ of chains satisfying a suitable notion of general position and such that the inclusion is a quasi-isomorphism. The chain-level intersection product is then a chain map  $\mu_k: G^k_*(M)\to C_*(M)$. 
 In arity $2$, one thus obtains a homology product as the composition
\begin{equation}\label{E: hom pairing}
H_*(M)\otimes H_*(M)\to H_*(C_*(M)\otimes C_*(M))\xleftarrow{\cong} H_*(G^2_*(M))\xr{\mu_2}H_{*-n}(M).\end{equation}
A similar construction for intersection chains on pseudomanifolds was given in \cite{GBF18}. Incidentally, data of the form $C_*\xleftarrow{q.i.}D_*\to E_*$, with q.i. denoting a quasi-isomorphism, constitutes the data of a morphism in the derived category of chain complexes, and so having our intersection product in this format is perfect for  obtaining an intersection product in the derived category of sheaves in \cite{GBF30}. 

\paragraph{A correction of the Goresky-MacPherson product.}
Another feature of this paper is that we correct a minor error in the Goresky-MacPherson intersection product \cite{GM1}. The issue occurs in \cite[Appendix]{GM1} in generalizing Dold's version of Poincar\'e  duality \cite[Proposition VIII.7.2]{Dold} to pseudomanifolds. We will explain the error in \cite{GM1}, provide a corrected proof of this pseudomanifold duality isomorphism, and explore some of its  properties that we will need. This material can be found in Section \ref{S: dualities}, following a detailed treatment of some properties of the Dold duality isomorphism. Beyond fixing the existing error, we hope that such a detailed discussion will be a useful expository addition to the literature.

\paragraph{Outline.}
Sections \ref{S: background}, \ref{S: chains}, \ref{S: cross}, and \ref{S: IC} of the paper contain review and foundational development of, respectively, PL stratified pseudomanifolds, PL chains, the PL cross product, and PL intersection chains. 
Section \ref{S: Dold} contains our study of Dold duality for manifolds, which we then use in Section \ref{S: GM} to provide our correction to the proof of Goresky-MacPherson duality for pseudomanifolds, which is in turn needed for the construction of the intersection product. 

Our development of the intersection product properly begins in Section \ref{S: pairings}. We construct the intersection product as a chain map on relative chains $\mf G_*^{\infty}(X,\Sigma)\to C_*^{\infty}(X,\Sigma)$, where $C_*^{\infty}(X,\Sigma)$ denotes the complex of locally-finite relative PL chains on $X$, the subspace $\Sigma$ is the locus of singular points of $X$, and $\mf G_*^{\infty}(X,\Sigma)\subset C_*^{\infty}(X,\Sigma)\otimes C_*^{\infty}(X,\Sigma)$ is the domain complex consisting of chains satisfying the needed stratified general position requirements. In particular, if $X=M$ is a manifold (unstratified), then $\Sigma$ is empty and we obtain an intersection map  $\mf G_*^{\infty}(M)\to C_*^{\infty}(M)$. These products generalize the products of compact chains from \cite{McC, GBF18}.  The 
intersection chain version of the product is derived as a consequence in Section \ref{S: int pairings}, and we show in Proposition \ref{P: compact} that in the compact setting it generalizes our product from \cite{GBF18} and the Goresky-MacPherson intersection product \cite{GM1}.

In the final section, Section \ref{S: cup dual}, we show that the work of this paper can be applied to give a short and direct proof that the Goresky-MacPherson homology product of \cite{GM1}, i.e.\ that induced by the chain-level intersection product,  is Poincar\'e dual to the cohomology cup product  on compact orientable PL manifolds. As we have noted, such a result is certainly not unexpected and it does not even employ the full power of our non-compact machinery, but the only other thorough proof of which the author is aware is that developed in \cite{GBF30} using some heavy sheaf-theoretic machinery. It is therefore quite reasonable to seek a proof using only standard algebraic topology techniques, and that is provided here. 
We also utilize this argument to confirm our above claim that the Goresky-MacPherson intersection product on the homology of manifolds agrees with Dold's homology product from \cite[Section VIII.13]{Dold}.
 Quite surprisingly, however, we will demonstrate that the argument given for manifolds does not seem to extend to a proof of the analogous duality between cup and intersection products for intersection (co)homology on pseudomanifolds, leaving an open question. We conclude that it seems that sheaf-theoretic techniques are, for now, necessary for pseudomanifolds, further motivating the work here for application in \cite{GBF30}.

\paragraph{A note on relations to past work.}
The construction of an intersection product of not-necessarily-compact intersection chains on not-necessarily-compact pseudomanifolds and the resulting map in the derived category of sheaf complexes was sketched in \cite[Remark 4.7]{GBF18}. This construction relied on a pseudomanifold version of a Poincar\'e duality isomorphism between cohomology and locally finite homology on non-compact manifolds provided by Spanier \cite{Sp93}. However, the  generalization to pseudomanifolds was not provided in detail in \cite{GBF18}. In any case, this approach to intersection products remains somewhat unsatisfactory, as Spanier's isomorphisms do not utilize a cap product with a fundamental class but rather a dual approach using Thom classes. Consequently, there is not an obvious direct way to compare the intersection product proposed in \cite{GBF18} with those of Goresky-MacPherson, McClure, and the author on compact spaces \cite{GM1, McC, GBF18}, as we do here. Furthermore, naturality with respect to restriction to open subsets was never formally verified in \cite{GBF18}. 
This paper thus remedies several lacunae from \cite{GBF18}.  

In fact, in our approach here to noncompact intersection products, we do not use at all a global version of Poincar\'e duality for noncompact chains as suggested in \cite{GBF18}. 
Rather, we show that it is possible to ``patch together'' local intersection information obtained on compact subsets. This method has the benefit of making it clear that intersection of chains is completely governed by local information, a fact that is not always transparent in the previous formulations of intersection products, even for compact chains. It is this local nature that makes the compatibility with restrictions to open subsets conceptually obvious, though there are still details to check to provide a complete proof.

\paragraph{A note on conventions.} In \cite{GBF18}, following McClure's  conventions in \cite{McC}, the chain complexes incorporate some degree shifts in order for the intersection product to be a chain map of degree $0$. Additionally, intersection products with more than two input  tensor factors were considered, with the goal of studying partial algebra structures. In the present paper, we  simplify somewhat by restricting to the more traditional two input factors and without shifts, so that our products will be chain maps of degree $-\dim(X)$. The reader should have little difficulty placing the development here back into the context of \cite{GBF18} if desired. We also generalize somewhat from the intersection chains used in  \cite{GBF18}; there, all perversities were assumed to be traditional perversities, i.e.\ those satisfying the original definition of Goresky and MacPherson in \cite{GM1}. Here, we will take advantage of the more  general notion of  perversity, which results in some simplifications. In particular, a perversity will be any function from the set of singular strata to $\Z$; additionally, 
our pseudomanifolds will be allowed codimension one strata. See \cite{GBF26} for a survey or \cite{GBF35} for full details. In this context, it makes sense to think of the PL intersection chain complexes  as subcomplexes not of $C_*(X)$ but of $C_*(X,\Sigma)$, explaining our setting of the intersection product in this context. Hence we will mostly work with relative chain complexes. We note, however, that when $\Sigma=\emptyset$, i.e.\ when $X$ is a PL manifold, then $C_*(X,\Sigma)=C_*(X)$, and we recover a product on absolute chains.

\section{PL stratified pseudomanifolds}\label{S: background}

All work in this paper is done in the piecewise linear (PL) category of topological spaces unless noted otherwise; we  refer the reader to \cite[Section 2.5 and Appendix B]{GBF35} for a review suited to our work here or to  \cite{RS} or \cite{HUD} as more thorough references. 
In this section, we provide a quick review of the essential definitions concerning pseudomanifolds. A review of intersection chains and intersection homology can be found in Section \ref{S: IC}, with \cite{GBF35} providing a thorough reference for all this material.

For a compact PL space $K$, we let $c(K)$ denote its \emph{open cone}; if $vK$ is the usual PL cone \cite[page 2]{RS} with vertex $v$, then $c(K)=(vK)-K$. Note 
that $c(\emptyset)=\{v\}$. If $K$ is a filtered PL space, meaning that $K$ is endowed with a family of closed PL subspaces
$$K=K^n\supset K^{n-1}\supset \cdots \supset 
K^0\supset K^{-1}=\emptyset,$$ 
we define $c(K)$ to be the filtered space with
$(c(K))^i=c(K^{i-1})$ for $i\geq 0$ and $(c(K))^{-1}=\emptyset$.
The definition of stratified pseudomanifold is inductive on the
dimension:

\begin{definition}\label{D: pseudomanifold}
A \emph{$0$-dimensional PL stratified pseudomanifold} $X$ is a  discrete set of points 
with the trivial filtration $X=X^0\supset X^{-1}=\emptyset$.

An $n$-dimensional \emph{PL stratified  pseudomanifold}
$X$ is a PL space filtered by closed PL subsets

\begin{equation*}
X=X^n\supset X^{n-1} \supset X^{n-2}\supset \cdots \supset X^0\supset X^{-1}=\emptyset
\end{equation*}
such that
\begin{enumerate}
\item $X-X^{n-1}$ is dense in $X$, and
\item for each point $x\in X^i-X^{i-1}$ there exist a neighborhood
$U$ of $x$, a  \emph{compact} $n-i-1$ dimensional 
PL stratified    pseudomanifold  $L$, and a PL  homeomorphism
\begin{equation*}
\phi: \R^i\times cL\to U
\end{equation*}
that takes $\R^i\times c(L^{j-1})$ onto $X^{i+j}\cap U$ for all $0\leq j\leq n-i-1$. A neighborhood $U$ with
this property is called \emph{distinguished} and $L$ is called a \emph{link} of
$x$. 
\end{enumerate}
\end{definition}

The $X^i$ are called \emph{skeleta}. As the number of skeleta is finite, there exist triangulations of $X$ with respect to which each $X^i$ is a subcomplex \cite[Lemma 2.5.12]{GBF35}; we will refer to such triangulations as being \emph{compatible with the stratification} and assume that all triangulations satisfy this requirement. In general, we say that a triangulation $T$ is \emph{compatible} with a PL subspace $Y$ if some subcomplex of $T$ triangulates $Y$. Note that we abuse notation by referring to the ``triangulation $T$'' without referring separately to the triangulating simplicial complex and the homeomorphism taking the simplicial complex to the space being triangulated.

We write $X_i$ for $X^i-X^{i-1}$; this 
is a PL $i$-manifold that may be empty. We refer to the connected components of 
the various $X_i$ as  \emph{strata}. 
 If $L$ and $L'$ are links of points in the same stratum then they are PL homeomorphic \cite[Lemma 2.5.18]{GBF35}. If $\dim(X)=n$ then a stratum $\mc Z$ that is a subset of $X_n$ 
is called a \emph{regular stratum}; otherwise it is called a \emph{singular 
stratum}.  Note that codimension $1$ strata are allowed. The union of the singular strata, which is identical to $X^{n-1}$, is often denoted $\Sigma$, or $\Sigma_X$. The stratified pseudomanifold $X$ is considered oriented if $X-\Sigma$ is given an orientation as a PL manifold. 

\section{Dold duality for manifolds and Goresky-MacPherson duality for pseudomanifolds}\label{S: dualities}

\subsection{Dold duality}\label{S: Dold}
In this section we consider Dold's Poincar\'e duality isomorphism of \cite[Proposition VIII.7.2]{Dold} in order to obtain some properties we will need and that are not set out explicitly in \cite{Dold}. According to \cite[Appendix]{GM1}, this isomorphism appears in Whitehead \cite{Wh62} but was first proven by Dold. 
The isomorphism will be used below to construct the Goresky-MacPherson duality isomorphism for pseudomanifolds, which in turn will be used for our intersection product. The results in this section apply for topological manifolds.

The following is Dold's duality isomorphism, slightly modified:

\begin{theorem}[Dold]
If $L\subset K$ are compact subsets of an oriented topological $n$-manifold $M$, then there is an isomorphism $\check H^i(K,L)\to H_{n-i}(M-L,M-K)$ induced by the cap product with a fundamental class  $\Gamma_K\in H_n(M,M-K)$.
\end{theorem}

Here $\check H^*(K,L)$ is the \v{C}ech cohomology of the pair $(K,L)$, which is isomorphic to the direct limit $\dlim H^*(V,W)$ as $(V,W)$ ranges over open neighborhood pairs of $(K,L)$ and the maps of the direct system are induced by restrictions to subspaces. See \cite[Section VIII.6]{Dold}. 

More specifically, the Dold duality isomorphism is constructed as follows: First let $(V,W)$ be a pair of open neighborhoods of $(K,L)$ and consider the composite 
\begin{equation}\label{E: Dold duality}
H^i(V,W)\xr{\cong} H^i(V-L,W-L)\xr{\frown j_*^{W,K}\Gamma_K}H_{n-i}(V-L,V-K)\xr{\cong}
H_{n-i}(M-L,M-K).
\end{equation}
The unlabeled maps are induced by inclusions, the leftmost being the excision isomorphism with $L$ being excised and the rightmost being the excision isomorphism with $M-V$ being excised. The class  $\Gamma_K\in H_n(M,M-K)$ is the fundamental class of $M$ over $K$ (see \cite[Lemma 3.27]{Ha} or \cite[Chapter 8]{GBF35}), and $j_*^{W,K}$ is the composition
\begin{equation}\label{E: Dold fund}
H_*(M,M-K)\to H_*(M,(M-K)\cup W)\xleftarrow{\cong }H_*(V-L, (V-K)\cup (W-L)),
\end{equation}
the leftward map being an excision isomorphism that excises the set $(M-V)\cup L$. Dold's isomorphism is the direct limit of the composition \eqref{E: Dold duality}  over all open neighborhoods $(V,W)$ of $(K,L)$. See \cite{Dold} for the proof that this is, in fact, an isomorphism.  This is Dold's version of Poincar\'e duality, and the reader may be more convinced of the plausibility of this claim by drawing some pictures, which, with the help of the right homotopy equivalences, can make the cap product appearing here look a lot like a more common form of Lefschetz duality.

So that we can later utilize Dold's isomorphism to obtain chain maps of the appropriate degree, we adopt the following definition:

\begin{definition}\label{D: Dold duality}
Define the \emph{Dold duality isomorphism} $\mf D:\check H^i(K,L)\to H_{n-i}(M-L,M-K)$ to be $(-1)^{in}$ times the direct limit of the composition \eqref{E: Dold duality}. The extra sign is necessary to be consistent with duality maps being chain maps of the appropriate  degree at the chain level; see \cite[Section 1.3]{GBF18} or \cite[Remark 8.2.2]{GBF35} for details.
\end{definition}

We demonstrate the naturality of the Dold duality isomorphism with respect to inclusion maps in the variables $K$, $L$, and $M$.

\begin{proposition}\label{P: natural Dold}
Let $(K,L)\subset (K',L')$ be  pairs of compact subsets of an oriented topological $n$-manifold $M'$ contained in a larger oriented topological $n$-manifold $M$, i.e.\ $M'\subset M$. Then there is a commutative diagram
\begin{diagram}
\check H^i(K,L)&\lTo& \check H^i(K',L')\\
\dTo^{\mf D}&&\dTo^{\mf D}\\
H_{n-i}(M-L,M-K)&\lTo& H_{n-i}(M'-L',M'-K'),
\end{diagram}
in which the horizontal maps are induced by inclusions and the vertical maps are  Dold duality isomorphisms.
\end{proposition}
\begin{proof}
Let $(V,W)$ be a pair of open neighborhoods of $(K,L)$, and, similarly, let $(V',W')$ be a pair of open neighborhoods of $(K',L')$.
We further suppose that $(V,W)\subset (V',W')$.

Now consider the following diagram:

\begin{diagram}
H^i(V,W)&\lTo& H^i(V',W')&\lTo^=&H^i(V',W')\\
\dTo&&\dTo&&\dTo\\
H^i(V-L,W-L)&\lTo& H^i(V'-L,W'-L)&\rTo&H^i(V'-L',W'-L')\\
\dTo^{\frown j_*^{W,K}\Gamma_K}&&\dTo^{\frown j_*^{W',K}\Gamma_K}&&\dTo^{\frown j_*^{W',K'}\Gamma_{K'}}\\
H_{n-i}(V-L,V-K)&\rTo& H_{n-i}(V'-L,V'-K)&\lTo&H_{n-i}(V'-L',V'-K')\\
\dTo&&\dTo&&\dTo\\
H_{n-i}(M-L,M-K)&\lTo^=& H_{n-i}(M-L,M-K)&\lTo&H_{n-i}(M'-L',M'-K').
\end{diagram}
Here, the unlabeled maps are induced by inclusions, $\Gamma_K\in H_n(M,M-K)$ and $\Gamma_{K'}\in H_n(M',M'-K')$ are orientation classes (see \cite[Lemma 3.27]{Ha}), while $j_*^{W,K}$ is the composite \eqref{E: Dold fund}.
 The maps $j_*^{W',K}$ and $j_*^{W',K'}$ are defined analogously. The direct limit over all such $(V,W)\supset (K,L)$ of the maps down the left side of the diagram is, up to sign $(-1)^{in}$, the  Dold isomorphism $\check H^i(K,L)\to
H_{n-i}(M-L,M-K)$, and similarly the direct limit over all such $(V',W')\supset (K',L')$ of the maps down the right side is the Dold isomorphism $\check H^i(K',L')\to H_{n-i}(M'-L',M'-K')$ up to the same sign. That these limits of maps are well-defined is demonstrated in \cite{Dold}. 

Let us observe that the diagram commutes. This is immediate for the squares not involving cap products. For the squares involving cap products, the commutativity is due to the naturality of the cap product \cite[\S VII.12.6]{Dold}. For the square on the left, we use naturality with respect to the inclusion map of triples $(V-L; V-K, W-L)\to (V'-L; V'-K, W'-L)$, and for the square on the right, this we use naturality with respect to  the inclusion map of triples $(V'-L'; V'-K', W'-L')\to (V'-L; V'-K, W'-L)$. To utilize the naturality, we observe that the images of $j_*^{W,K}\Gamma_K$ and $j_*^{W',K'}\Gamma_{K'}$ under the respective  maps induced by inclusion $H_n(V-L, (V-K)\cup( W-L))\to H_n(V'-L, (V'-K)\cup( W'-L))$ and  $H_n(V'-L', (V'-K')\cup( W'-L'))\to H_n(V'-L, (V'-K)\cup( W'-L))$ are each indeed $j_*^{W',K}\Gamma_K$. This follows from an easy commutative diagram argument, noting that the image of $\Gamma_{K'}$ in $H_n(M,M-K)$ is $\Gamma_K$, by \cite[Lemma 3.27]{Ha}. 

An easy diagram chase now shows that the outer square of the diagram commutes, yielding the commutative square

\begin{diagram}
H^i(V,W)&\lTo&H^i(V',W')\\
\dTo&&\dTo\\
H_{n-i}(M-L,M-K)&\lTo&H_{n-i}(M'-L',M'-K'),
\end{diagram}
and taking the direct limit over $(V,W)\supset (K,L)$ yields a diagram 
\begin{diagram}
\check H^i(K,L)&\lTo&H^i(V',W')\\
\dTo^{\mf D}&&\dTo\\
H_{n-i}(M-L,M-K)&\lTo&H_{n-i}(M'-L',M'-K'),
\end{diagram}
where the lefthand vertical map is Dold's isomorphism. Lastly, taking the direct limit over $(V',W')\supset (K',L')$ gives the diagram 
\begin{diagram}
\check H^i(K,L)&\lTo&\check H^i(K',L')\\
\dTo^{\mf D}&&\dTo^{\mf D}\\
H_{n-i}(M-L,M-K)&\lTo&H_{n-i}(M'-L',M'-K').
\end{diagram}
\end{proof}

We also need to know how Dold's isomorphism interacts with boundary maps:

\begin{proposition}\label{P: Dold boundary}
Let  $L\subset K \subset J$ be compact subsets of an oriented topological $n$-manifold $M$.  
The following diagram commutes up to the sign $(-1)^n$:
\begin{diagram}
\check H^i(K,L)&\rTo^{d^*}&\check H^{i+1}(J,K)\\
\dTo^{\mf D}&&\dTo^{\mf D}\\
H_{n-i}(M-L,M-K)&\rTo^{\bd_*}&H_{n-i-1}(M-K,M-J).
\end{diagram}
\end{proposition}
\begin{proof}
Consider the following diagram, in which $\mc D$ is the Dold isomorphism but without our added sign (i.e.\ $\mc D(\alpha)=(-1)^{|\alpha|n}\mf D(\alpha)$):

\begin{diagram}
\check H^i(K,L)&\rTo& \check H^i(K)&\rTo^{d^*}&\check H^{i+1}(J,K)\\
\dTo^{\mc D}&&\dTo^{\mc D}&&\dTo^{\mc D}\\
H_{n-i}(M-L,M-K)&\rTo&H_{n-i}(M,M-K)&\rTo^{\bd_*}&H_{n-i-1}(M-K,M-J).
\end{diagram}
The unlabeled maps are induced by inclusions.
Each of the squares commute by the arguments in Case 6 in the proof of \cite[Proposition VIII.7.2]{Dold}. The bottom horizontal composition is equal to the boundary map in the exact sequence of the triple $(M-L,M-K,M-J)$, as we see by the natural transformation to the  sequence of the triple $(M,M-K,M-J)$:

\begin{diagram}
&\rTo&H_{n-i}(M-L,M-J)&\rTo&H_{n-i}(M-L,M-K)&\rTo^{\bd_*}&H_{n-i-1}(M-K,M-J)&\rTo\\
&&\dTo&&\dTo&&\dTo^=\\
&\rTo&H_{n-i}(M,M-J)&\rTo&H_{n-i}(M,M-K)&\rTo^{\bd_*}&H_{n-i-1}(M-K,M-J)&\rTo
\end{diagram}
Similarly, the composition across the top is the cohomology coboundary map of the triple $(J,K,L)$ via

\begin{diagram}
&\rTo&\check H^i(J,L)&\rTo&\check H_i(K,L)&\rTo^{d^*}&\check H^{i+1}(J,K)&\rTo&\\
&&\dTo&&\dTo&&\dTo^=\\
&\rTo&\check H^i(J)&\rTo&\check H^i(K)&\rTo^{d^*}&\check H^{i+1}(J,K)&\rTo&.
\end{diagram}

So, if $\alpha\in \check H^i(K,L)$ then 
\begin{align*}
\mf D(d^*(\alpha))&=(-1)^{(i+1)n}\mc D (d^*(\alpha))\\
&=(-1)^{(i+1)n}\bd_*(\mc D(\alpha))\\
&=(-1)^n\bd_*(\mf D(\alpha)).
\end{align*}
\end{proof}

\subsection{Goresky-MacPherson duality}\label{S: GM}

In the appendix to \cite{GM1}, Goresky and MacPherson provide a generalization of the Dold duality isomorphism to PL stratified pseudomanifolds. However, there is an error in their argument. In this section, we prove a slight modification that does not claim the full generality of the original Goresky-MacPherson statement but that is sufficient both for the applications here and in \cite{GM1}. 

The setting in \cite[Appendix]{GM1} assumes a compact oriented   $n$-dimensional PL pseudomanifold $X$ with singular locus $\Sigma$ and 
with  $A\subset B$ two  \emph{constructible subsets}, i.e.\ unions of interiors of simplices in some triangulation of $X$. It is also assumed that $B-A\subset X-\Sigma$ or, equivalently, that $A\cap \Sigma=B\cap \Sigma$. The claimed duality isomorphism is of the form $H^i(B,A)\to H_{n-i}(X-A,X-B)$, induced by the cap product with the fundamental class. The first step in constructing this isomorphism, which occurs in the proof of the lemma on page 162, 
is ``$H^i(B,A)\cong H^i(B\cup \Sigma, A\cup \Sigma)$ by excision of $\Sigma-(A\cap \Sigma)$.'' However, this is not always an isomorphism. 

For example, consider $B=X-\Sigma$ and $A=\emptyset$, in which case $(X-A,X-B)=(X,\Sigma)$; this case is of practical importance to the intersection theory as elements of $H_n(X,\Sigma)$ arise as homological representatives of  the fundamental class of $X$ in the Goresky-MacPherson version of the intersection pairing. Note that the requirement $B-A=X-\Sigma\subset X-\Sigma$ is satisfied trivially here, and $B$ is constructible as $\Sigma$ is a subcomplex in any triangulation compatible with the stratification. But then the excision isomorphism would be $H^i(X-\Sigma)\cong H^i(X, \Sigma)$, which is false, for example, when $i=0$, $\Sigma$ is neither empty nor all of $X$, and $X$ is connected.  

We will give a construction of the Goresky-MacPherson isomorphism below as Proposition \ref{P: GM dual}. Our approach 
utilizes essentially the same basic ideas as that of Goresky and MacPherson but avoids the problematic excision. We  simplify matters somewhat by limiting ourselves to the case where, in the Goresky-MacPherson notation, $X-A$ and $X-B$ are subcomplexes of some triangulation of $X$, as this is the only situation required for the construction of intersection products both below and in \cite{GM1}. Note that the example of the prior paragraph satisfies this condition, so the limitation does not obviate the error in \cite{GM1}. 

\begin{remark}\label{R: GM bdbd}
 There is another minor logical problem with the Goresky-MacPherson construction that is cured using the approach of McClure in \cite{McC}, which is emulated in \cite{GBF18} and below:  In Section 2.1 of \cite{GM1}, Goresky and MacPherson define a notion of \emph{dimensional transversality} and define their intersection product on pairs of chains $C,D$ such that the following pairs are dimensionally transverse: $(C,D)$, $(C,\bd D)$, and $(\bd C,D)$. If we let $C\pf D$ denote the Goresky-MacPherson intersection product\footnote{This is not the notation in \cite{GM1}, where $\pf$ serves a different purpose.},  it is then claimed that $\bd (C\pf D)=(\bd C)\pf D+(-1)^{n-|C|}C\pf (\bd D)$. The trouble is that in order for, say, $(\bd C)\pf D$ to be defined, one must then \emph{also} know that the pair $(\bd C,\bd D)$ is dimensionally transverse, which is not a part of the initial assumption. This does not become a serious problem in \cite{GM1}, however, as the primary interest there is in the intersection of cycles. 
\end{remark}

In the following proposition, we do not require that $X$ be a pseudomanifold, but merely that we have a compact PL pair $(X,S)$ such that  $X-S$ is an  $n$-dimensional oriented manifold.

\begin{proposition}\label{P: GM dual}
Let $(X,{S})$ be a compact  PL space pair such that $X-{S}$ is an oriented $n$-dimensional (topological) manifold, and let $L\subset K\subset X$ be compact PL subspaces, such that $S\subset L$. Then there is an isomorphism $\D:H^i(X-L,X-K)\to H_{n-i}(K,L)$ composed of excisions, isomorphisms induced by inclusions, and the Dold duality isomorphism. 
\end{proposition}
\begin{proof}
As ${S}$, $L$, and $K$ are PL subspaces of $X$, there is a triangulation $T$ of $X$ with respect to which these subspaces arise as subcomplexes \cite[Addendum 2.12]{RS}. Furthermore, by replacing $T$ with a subdivision, if necessary, we may assume that each is a full subcomplex of $T$, using \cite[Lemma 3.3]{RS}. 

Suppose $Z$ is the underlying space of a full subcomplex of $T$. Let $C_Z$ be the the union of all simplices in $T$ that are disjoint from $Z$.  Then, by \cite[Lemma 70.1]{MK} and its proof, the subcomplex corresponding to $C_Z$ is also a full subcomplex of $T$, $Z$ is a deformation retract of $X-C_Z$, and $C_Z$ is a deformation retract of $X-Z$. Furthermore, if, as in \cite[Section 72]{MK} (modifying that notation slightly),  we let $St(Z)$  be the union of the interiors of all simplices of $T$ that have a face in $Z$, then $X-C_Z=St(Z)$.  

Note that, as $X$ is compact, each $C_Z$ is also compact. Also, as ${S}\subset L\subset K$, both sets $C_L$ and $C_K$ are subsets of the manifold $X-{S}$. 

The claimed isomorphism of the lemma is the composite of the following isomorphisms:

\begin{align}
H^i(X-L,X-K)&\xr{\cong} H^i(C_L,C_K)&\text{homotopy equivalences}\notag\\
&\xr{\mf D} H_{n-i}((X-{S})-C_K, (X-{S})-C_L) &\text{Dold}\notag\\
&= H_{n-i}((X-C_K)- {S}, (X-C_L)-{S}) &\text{set equality} \label{E: Dold}\\
&= H_{n-i}(St(K)-{S}, St(L)-{S})&\text{see above}\notag\\
&\xr{\cong}H_{n-i}(St(K), St(L))&\text{by excision, see below}\notag\\
&\xleftarrow{\cong} H_{n-i}(K, L) &\text{homotopy equivalences.}\notag
\end{align}
 
For the excision in the fifth line, which excises ${S}$, we note that ${S}$ is a closed subcomplex contained in the interior of $St(L)$ by construction. 

The map labeled ``Dold'' is the duality isomorphism reviewed in the preceding section. Note that $\check H^i(C_L\cup{S},C_K\cup {S})\cong  H^i(C_L\cup{S},C_K\cup {S})$ by \cite[Proposition CIII.6.12]{Dold}, as both spaces are ENRs as subcomplexes of compact simplicial complexes (in fact any finite CW complex is an ENR \cite[Corollary A.10]{Ha}). 

The construction of $\D$ as stated depends on the triangulation $T$. But if $T'$ is a sufficiently fine refinement of $T$ we obtain complement sets $C'_L$, $C'_K$ and star sets $St'(K)$, $St'(L)$ with $C_K\subset C'_K$, $C_L\subset C'_L$, $St(K')\subset St(K)$, and $St(L')\subset St(L)$. We can then easily form a ladder diagram comparing our first definition of $\D$ with that using $T'$. Such a diagram commutes by the naturality of inclusions and of the Dold map $\mf D$, using Proposition \ref{P: natural Dold}. As any two triangulations of $X$ will have a common refinement, this demonstrates that $\D$ does not depend on the choice of $T$. 
\end{proof}

We now prove some properties of the duality isomorphism $\D$.

\begin{lemma}\label{L: natural dual}
The duality isomorphism of Lemma \ref{P: GM dual} is natural in the sense that if  $(K,L)\subset (K',L')$ are pairs of compact PL subspaces satisfying the criteria of Lemma \ref{P: GM dual}, then there is a commutative diagram
\begin{diagram}
H^i(X-L,X-K)&\rTo& H^i(X-L',X-K')\\
\dTo^\D&&\dTo^\D\\
H_{n-i}(K,L)&\rTo& H_{n-i}(K',L'),
\end{diagram}
in which the horizontal maps are induced by inclusions and the vertical maps are the isomorphisms of Lemma \ref{P: GM dual}.
\end{lemma}
\begin{proof}
By the observation at the end of the proof of Lemma \ref{P: GM dual}, we may use any triangulation with respect to which all our subsets are subcomplexes. 
With the exception of the Dold duality isomorphism, it is then straightforward to verify that all of the maps in the sequence of isomorphisms in Proposition \ref{P: GM dual} commute with appropriate inclusion maps. The commutativity for the Dold isomorphisms is provided by Lemma \ref{P: natural Dold}. 
\end{proof}

There are two other observations we will need concerning the duality isomorphism $\D$. As stated, the lemma and its construction utilize the singular set $S$ as part of the input data. It will be useful to know that the isomorphism is in fact independent of $S$, so long as $S\subset L$ and $X-S$ is an oriented $n$-manifold.

\begin{lemma}\label{L: sing indep}
The duality isomorphism of Proposition \ref{P: GM dual} does not depend on the choice of compact set $S$ such that $S\subset L\subset  X$ and such that $X-S$ is an oriented $n$-manifold. 
\end{lemma}
\begin{proof}
Let $S'$ be an alternative such subspace. We may assume that $S'\subset S$, for if not then the claim follows by comparing the isomorphism determined by each of $S$ and $S'$ to that determined by $S\cap S'$. So, assuming that $S'\subset S$, we have the following diagram, continuing to use the notation of the proof of Proposition \ref{P: GM dual}:
\begin{diagram}
H^i(C_L,C_K)&\rTo^=& H^i(C_L,C_K)\\
\dTo^{\mf D}&&\dTo^{\mf D}\\
H_{n-i}((X-{S})-C_K, (X-{S})-C_L)&\rTo&H_{n-i}((X-{S'})-C_K, (X-{S'})-C_L)\\
\dTo^=&&\dTo^=\\
H_{n-i}((X-C_K)- {S}, (X-C_L)-{S}) &\rTo &H_{n-i}((X-C_K)- {S'}, (X-C_L)-{S'})\\
\dTo^=&&\dTo^=\\
H_{n-i}(St(K)-{S}, St(L)-{S})&\rTo&H_{n-i}(St(K)-{S'}, St(L)-{S'})\\
\dTo^\cong&&\dTo^\cong\\
H_{n-i}(St(K), St(L))&\rTo^=&H_{n-i}(St(K), St(L))
\end{diagram}
The top square commutes by Proposition \ref{P: natural Dold} and the remaining squares clearly commute. Commutativity of the whole diagram, together with the definition of the map in  Proposition \ref{P: GM dual}, demonstrates the independence of choice of $S$. 
\end{proof}

We will also need that the duality isomorphism of Proposition \ref{P: GM dual} is natural with respect to ``expansion of the singularity.'' In other words, suppose that we enlarge $X$ to a complex $X\cup T$ such that $X\cap T\subset S$. Then, as $S\subset L$, we have $H_i(K,L)\cong H_i(K\cup T, L\cup T)$, while $(X-(L\cup T),X-(K\cup T))=(X-L,X-K)$. We have the following compatibility:

\begin{lemma}\label{L: expansion}
Let $(X,{S})$ be a compact  PL space pair such that $X-{S}$ is an oriented $n$-dimensional manifold, and let $L\subset K\subset X$ be compact PL subspaces such that $S\subset L$. Let $X\cup T$ be a compact PL space with $X\cap T\subset S$.  Then the following diagram commutes, with each vertical map being the isomorphism of Proposition \ref{P: GM dual}: 
\begin{diagram}
H^i(X-L,X-K)&\rTo^=&H^i(X-(L\cup T),X-(K\cup T))\\
\dTo^\D&&\dTo^\D\\
H_{n-i}(K,L)&\rTo^\cong &H_{n-i}(K\cup T,L\cup T).
\end{diagram}
\end{lemma}
\begin{proof}
This time we use  the commutative diagram
{\footnotesize
\begin{diagram}
H^i(C_L,C_K)&\rTo^=& H^i(C_L,C_K)\\
\dTo^{\mf D}&&\dTo^{\mf D}\\
H_{n-i}((X-{S})-C_K, (X-{S})-C_L)&\rTo^=&H_{n-i}(((X\cup T)-(S\cup T))-C_K, ((X\cup T)-(S\cup T))-C_L)\\
\dTo^=&&\dTo^=\\
H_{n-i}((X-C_K)- {S}), (X-C_L)-{S})) &\rTo^= &H_{n-i}(((X\cup T)-C_K)- (S\cup T)), ((X\cup T)-C_L)-(S\cup T))\\
\dTo^=&&\dTo^=\\
H_{n-i}(St_X(K)-{S}, St_X(L)-{S})&\rTo^=&H_{n-i}(St_{X\cup T}(K\cup T)-(S\cup T), St_{X\cup T}(L\cup T)-(S\cup T))\\
\dTo&&\dTo\\
H_{n-i}(St_X(K), St_X(L))&\rTo&H_{n-i}(St_{X\cup T}(K\cup T), St_{X\cup T}(L\cup T))\\
\uTo^\cong&&\uTo^\cong\\
H_{n-i}(K,L)&\rTo^\cong&H_{n-i}(K\cup T, L\cup T).
\end{diagram}}

Here the subscript on $St$ indicates in which space the star is taken. 
\end{proof}

The next lemma shows how the duality isomorphism $\D$ interacts with the connecting morphisms:

\begin{lemma}\label{L: GM dual boundary}
Let $(X,{S})$ be a compact  PL space pair such that $X-{S}$ is an oriented $n$-dimensional manifold, and let $J\subset L\subset K\subset X$ be compact PL subspaces such that $S\subset J$. Then the following diagram commutes up to $(-1)^n$:

\begin{diagram}
H^i(X-L,X-K)&\rTo^{d^*}&H^{i+1}(X-J,X-L)\\
\dTo^\D&&\dTo^\D\\
H_{n-i}(K,L)&\rTo^{\bd_*}&H_{n-i-1}(L,J).
\end{diagram} 
\end{lemma}
\begin{proof}
Consider the definition of the duality map of Proposition \ref{P: GM dual} and the evident ladder diagram connecting the maps of the definition for the pair $(K,L)$ with that for the pair $(L,J)$ via boundary/coboundary maps. It is clear from the standard naturality of the connecting maps in the long exact sequences that such a diagram commutes in all squares with the possible exception of the square involving the Dold duality map. But this square commutes up to $(-1)^n$ by Proposition \ref{P: Dold boundary}. 
\end{proof}

\section{PL chains}\label{S: chains}

In this section, we first review some fundamentals concerning PL chains and then review and develop some 
 useful identifications between groups of PL chains and certain homology groups. 

\subsection{PL chains}\label{S: PL chains}

Let $X$ be a locally compact PL space. Any triangulation $T$ of $X$ is locally finite \cite[Lemma 2.6]{MK}.   Let $c^{T,\infty}_i(X)$ be the group of locally finite $i$-chains on $X$ with respect to the triangulation $T$, i.e.\ the elements of  $c^{T,\infty}_i(X)$ are formal sums $\sum_\sigma a_\sigma \sigma$, where the sum is over \emph{all} oriented\footnote{More precisely,  each simplex comes equipped with two orientations $(\sigma,o)$ and $(\sigma,-o)$, and we impose in $c^{T,\infty}_i(X)$ the relation $-(\sigma,o)=(\sigma,-o)$. For the sake of writing chains as sums, we follow standard practice and choose the notation ``$\sigma$'' to denote a simplex with one of its orientations, arbitrarily chosen, and then allow the simplex to appear in the sum expression only with this orientation.}  $i$-simplices $\sigma$ of $X$ in $T$ and each $a_\sigma\in \Z$. These are called \emph{locally finite chains} or \emph{chains with closed support} or \emph{Borel-Moore chains}. The local-finiteness ensures that the boundary map is well defined in the usual way (see \cite{BoHab, GM2} or \cite[Section 4.1.3]{BaIH}), and thus $c^{T,\infty}_*(X)$ is a chain complex with homology groups $H^{T,\infty}_*(X)$. Relative chain complexes and relative homology are defined in the obvious way. If $T'$ is a subdivision of $T$, there is a subdivision map $c^{T,\infty}_*(X)\to c^{T',\infty}_*(X)$ that (formally) takes each $i$-simplex to the sum of the $i$-simplices contained in it with compatible orientations. 
We let $C^{\infty}_*(X)$ be the direct limit of the $c^{T,\infty}_i(X)$ over all compatible  triangulations of $X$ with maps induced by the subdivision maps. The corresponding homology groups are $H^{\infty}_*(X)$. 

Given a chain $\xi\in C_i^{\infty}(X)$, its support $|\xi|$ is defined as follows: write $\xi$ as an element  $\sum_\sigma a_\sigma \sigma\in c^{T,\infty}_i(X)$ for some triangulation $T$ and let $|\xi|$ be the union of the $i$-simplices $\sigma$ such that $a_\sigma\neq 0$. This definition is independent of the choice of $T$. If $\xi\in C_i^{\infty}(X,Y)$, for $Y$ a PL subspace of $X$, then we can again
write a representative for  $\xi$ as an element  $\sum_\sigma a_\sigma \sigma\in c^{T,\infty}_i(X)$ for some triangulation $T$ for which $Y$ is a subcomplex, then we let $|\xi|$ be the union of the  $i$-simplices $\sigma$ not contained in $Y$ and such that $a_\sigma\neq 0$. In other words, $|\xi|$ is the support of the minimal chain representing $\xi$ in $C_i^{\infty}(X,Y)$.

If we let $c^{T}_*(X)$ be the usual simplicial chain complex in which each chain is a finite sum of simplices with coefficients, then $c^{T}_*(X)\subset c^{T,\infty}_*(X)$. Taking subdivisions we obtain the PL chain complex with compact supports,  $C_*(X)\subset C_*^\infty(X)$.

\subsection{Useful lemmas}

By  \cite[\S 2.1]{GM2} and \cite[\S 1]{BoHab}, we have the following useful isomorphism. A detailed proof can be found in   \cite[Lemma 4]{McC} (compare also  \cite[\S 1.2]{GM1}) for the corresponding result concerning compactly supported chains, though the proof in the case of locally finite chains  is identical:

\begin{lemma}\label{L: useful}
Let $X$ be a PL space and let $B\subset A$ be closed PL subspaces of $X$ such that $\dim(A)=p$ and $\dim(B)<p$. Then 
\begin{enumerate}
\item there is an isomorphism  $\alpha_{A,B}$ from $H^{\infty}_p(A,B)$ to the abelian group
$$\mf C_p^{A,B}=\{\xi\in C^{\infty}_p(X)\mid |\xi|\subset A, |\bd \xi|\subset B\},$$ 

\item the following diagram commutes:
\begin{diagram}
H^{\infty}_p(A,B)&\rTo^{\alpha_{A,B}} &\mf C_p^{A,B}\\
\dTo^{\bd _*}&&\dTo^\bd\\
H^{\infty}_{p-1}(B)&\rTo^{\alpha_{B,\emptyset}} &\mf C_{p-1}^{B,\emptyset}\\
\end{diagram}

\end{enumerate}
\end{lemma}

\begin{proof}
As observed in \cite{McC},  $H^{\infty}_p(A,B)$ is  the $p$–th homology of the
complex $C^{\infty}_*(A, B)$, which is the quotient of the relative cycles
by the relative boundaries. But the set $\mf C_p^{A,B}$ is precisely the set of relative cycles,
while the set of relative boundaries is $0$ by the dimension hypothesis. The second part of the lemma follows by chasing the definitions. 
\end{proof}

In particular then, as noted in \cite[\S 1]{BoHab}, a chain $\xi\in C^{\infty}_p(X)$ is completely described by $|\xi|$, $|\bd \xi|$, and the class represented by $\xi$ in $H_p^\infty(|\xi|,|\bd \xi|)$.

\begin{remark}\label{R: useful coeff}
Suppose  $\xi\in \mf C_p^{A,B}$ and $T$ is a triangulation of $X$ such that $A$ and $B$ are subcomplexes and $\xi$ can be represented as an element of $c^{T,\infty}_i(X)$. Let $\sigma$ be an (oriented) $p$-simplex of $A$. Then the coefficient of $\sigma$ in the representation of $\xi$ in $c^{T,\infty}_i(X)$ can be recovered as the image of   $\alpha_{A,B}^{-1}(\xi)$ in $H_{p}^\infty(A,A-int(\sigma))\cong H_p(\sigma,\bd \sigma)\cong \Z$, the first isomorphism  by excision. The argument in this context for excision is identical to the standard simplicial proof, e.g.\ \cite[Theorem 9.1]{MK}. The claim follows by letting the   representation of $\xi$ in $c^{T,\infty}_i(X)$ also serve as the chain representative for $\alpha_{A,B}^{-1}(\xi)$ and then its image in $H_{p}^\infty(A,A-int(\sigma))$. So, in fact,  given an element  $[\xi]\in H^{\infty}_p(A,B)$ and a triangulation $T$ compatible with $A$ and $B$, we can in this way construct a representative of $\alpha_{A,B}([\xi])$ in $c^{T,\infty}_p(A)\subset c^{T,\infty}_p(X)$ by so determining its coefficient at each $p$-simplex of $A$. 
\end{remark}

In our application below, we will need to consider more general scenarios in which $A$ might have arbitrary dimensions but $A-B$ still has dimension $p$. For this we have the following lemma. We let $cl(Y)$ denote the closure of the subspace $Y$ in its ambient space.

\begin{lemma}\label{L: useful coeff big}
Let $X$ be a PL space and let $C\subset B\subset A$ be closed PL subspaces of $X$  such that $\dim(A-B)=p$ and $\dim(B-C)<p$, and let $T$ be a triangulation of $X$ with respect to which $A,B,C$ are subcomplexes. Then: 
\begin{enumerate}
\item  There is an isomorphism  $\bar \alpha_{A,B}$ from $H^{\infty}_p(A,B)$ to the abelian group
$$\bar{\mf{C}}_p^{A,B}=\{\xi\in C^{\infty}_p(X)\mid |\xi|\subset cl(A-B), |\bd \xi|\subset B\}.$$
If   $[\xi]\in H^{\infty}_p(A,B)$ is represented by a simplicial chain $\xi=\sum a_\sigma \sigma$ with $p$-simplices in $T$, then $\bar \alpha_{A,B}([\xi])=\sum_{\sigma\not\subset  B} a_\sigma \sigma$, with the sum being over $p$-simplices of $T$ not contained in $B$ (though necessarily $\sigma \subset A$).

\item If $[\xi]\in H^{\infty}_p(A,B)$ with $\bd\bar \alpha_{A,B}([\xi])=\sum b_\tau \tau$ and $\bd_*:H^\infty_p(A,B)\to H^\infty_{p-1}(B,C)$ is the connecting morphism, then $\bar \alpha_{B,C}(\bd_*([\xi]))=\sum_{\tau\not\subset  C} b_\tau \tau$. 
\end{enumerate}
\end{lemma}
\begin{proof}
Let $\bar A_B$ denote the closure of $A-B$. In the triangulation $T$, this is just the space of $p$-simplices of $A$ that are not contained in $B$. We  have  $\dim(\bar A_B)=p$ and also $\dim(\bar A_B\cap B)<p$, as $\bar A_B\cap B$ consists of those simplices of $B$ that are faces of the simplices of $A$ that are not contained in $B$. Thus, via Lemma \ref{L: useful}, we have an isomorphism between $H^\infty_p(\bar A_B,\bar A_B\cap B)$ and 
$\mf C_p^{\bar A_B,\bar A_B\cap B}$. 
There is also an excision isomorphism $H^\infty_p(\bar A_B,\bar A_B\cap B)\to H^\infty_p(A,B)$. 
Furthermore, $\mf C_p^{\bar A_B,\bar A_B\cap B}=\bar{\mf{C}}_p^{A,B}$ from the definitions, noting that if $\xi$ is a chain in $\bar A_B$, then $\bd \xi$ is in $B$ if and only if it is in $B\cap A_B$. 
We let $\bar \alpha_{A,B}:H^{\infty}_p(A,B) \to \bar{\mf{C}}_p^{A,B}$ be the composition of these isomorphisms.  In addition, if   $[\xi]\in H^{\infty}_p(A,B)$ is represented by  $\xi=\sum a_\sigma \sigma$, it is also represented by $\bar \xi=\sum_{\sigma\not\subset  B} a_\sigma \sigma$, as the piece of $\xi$ contained in $B$ is $0$ in $C^{\infty}_p(A,B)$. But then $\bar \xi\in \bar{\mf{C}}_p^{A,B}=\mf C_p^{\bar A_B,\bar A_B\cap B}$. It follows that $\bar \alpha_{A,B}([\xi])=\bar \xi$ from the proof of Lemma \ref{L: useful}.

For the second part of the lemma, if $[\xi]\in H^{\infty}_p(A,B)$ is represented by the chain  $\xi=\sum a_\sigma \sigma$, then, as above, it is also represented by $\bar \xi=\sum_{\sigma\not\subset  B} a_\sigma \sigma$. So $\bd_*([\xi])$ is represented by $\bd \bar \xi$, which is $\sum b_\tau \tau$ by assumption. The map $\bar \alpha_{B,C}$ thus takes $\bd_*([\xi])$ to $\sum_{\tau\not\subset  C} b_\tau \tau$ by the argument just above.
\end{proof}

\begin{remark}\label{R: useful coeff big}
As in Remark \ref{R: useful coeff}, given $[\xi]\in H_p^\infty(A,B)$ in the setting of Lemma \ref{L: useful coeff big}, as well as a   triangulation $T$ compatible with $A$ and $B$ and a $p$-simplex $\sigma$ of $cl(A-B)$,  we can recover the coefficient of $\sigma$ in the corresponding chain  $\bar \alpha^{A,B}([\xi])$ in $c^{T,\infty}_i(X)$ by looking at the image of $[\xi]$ in $H_{p}^\infty(A,A-int(\sigma))\cong H_p(\sigma,\bd \sigma)\cong \Z$. In this way, given   $[\xi]\in H^{\infty}_p(A,B)$ and a triangulation $T$ compatible with $A$ and $B$, we can construct the unique chain in $c^{T,\infty}_p(\bar A_B)\subset c^{T,\infty}_p(X)$  prescribed by the isomorphisms. 
\end{remark}

\section{Chain-level intersection products on pseudomanifolds}\label{S: pairings}

This section contains our development of the intersection product $$C_*^{\infty}(X,\Sigma)\otimes C_*^{\infty}(X,\Sigma)\supset \mf G_*^{\infty}(X,\Sigma)\xr{\mu} C_*^{\infty}(X,\Sigma)$$ for locally finite chains on not-necessarily-compact pseudomanifolds. Following some preparatory remarks, we define \emph{intersection coefficients} via \emph{local umkehr maps} in Section \ref{S: local umkehr}; these are used to define  a \emph{global umkehr map} in Section \ref{S: umkehr}. Section \ref{S: cross} concerns the chain cross product, and, finally, Section \ref{S: pairing} contains the definition of the intersection product.

The rough idea for constructing intersection products, going back at least to Dold in \cite[Section VIII.13]{Dold} and then used in \cite{McC, GBF18}, is the observation that if $A,B\subset X$ then a point $x\in X$ is in $A\cap B$ if and only if the point $(x,x)\in X\times X$ is contained in $A\times B$. When $A$ and $B$ are replaced by chains $\xi$ and $\eta$, this lets us utilize the well known chain cross product to consider $\xi\times \eta$ in $X\times X$. Then we pull back to $X$ itself via the diagonal map  $\Delta: X\to X\times X$ with $\Delta(x)=(x,x)$. This process is reminiscent of the construction of the cup product. However, as chains behave covariantly we do not have a pullback map $\Delta^*$ but must instead 
utilize a version of the  \emph{umkehr} or \emph{transfer} map $\Delta_!$, versions of which can be found for manifolds in \cite[Section VIII.10]{Dold} or \cite[Definition 6.11.2]{BRTG}. 

The cross the product will be reviewed below. We begin instead by outlining the construction to follow in 
Sections \ref{S: local umkehr} and \ref{S: umkehr}
of the umkehr chain  map $$\Delta_!:C_*^{\Delta,\infty}((X,\Sigma)\times (X,\Sigma))\to C^\infty_{*-n}(X,\Sigma),$$ generalizing previous constructions to not-necessarily-compact pseudomanifolds.
We use here the notation for products of pairs: $$(A,B)\times (C,D)=(A\times C, (A\times D)\cup (B\times C)).$$
Throughout this section, our $X$ is an oriented $n$-dimensional PL stratified pseudomanifold, not necessarily compact, and $C_*^{\Delta,\infty}((X,\Sigma)\times (X,\Sigma))$ is the subcomplex of chains of $C_*^{\infty}((X,\Sigma)\times (X,\Sigma))$ that are in general position with respect to the map $\Delta$.
 Explicitly, this means the following:

 \begin{definition}\label{D: Delta}
Let $X$ be an oriented $n$-dimensional PL stratified pseudomanifold. For a set $A\in X\times X$, let $A'=\Delta^{-1}(A)\subset X$. Then  $C_i^{\Delta,\infty}((X,\Sigma)\times (X,\Sigma))\subset C_i^{\infty}((X,\Sigma)\times (X,\Sigma)))$ is the subcomplex  consisting of those chains $\xi$ such that 
\begin{enumerate}
\item $\dim(|\xi|'-\Sigma)\leq i-n$

\item $\dim(|\bd\xi|'-\Sigma)\leq i-n-1$.
\end{enumerate}

Observe that this is indeed a chain subcomplex, as if $\xi_1,\xi_2\in C_i^{\Delta,\infty}((X,\Sigma)\times (X,\Sigma))$ then so is $\xi_1+\xi_2$, as $|\xi_1+\xi_2|\subset |\xi_1|\cup |\xi_2|$, so $|\xi_1+\xi_2|'\subset |\xi_1|'\cup |\xi_2|'$ and similarly for $\bd(\xi_1+\xi_2)$. The second condition guarantees that if $\xi\in C_i^{\Delta,\infty}((X,\Sigma)\times (X,\Sigma))$ then so is $\bd \xi$. 
\end{definition}

\begin{remark}\label{R: weak GP}
These are weaker conditions than what were required in \cite[Definition 4.2]{GBF18}, where there were also requirements on the dimensions of intersection with the singular strata in both $X\times X$ and $X$. Our use here of the relative chain complex ultimately makes these additional assumptions unnecessary for the purposes of defining an intersection product on relative chains. However, we will impose some additional constraints below in Section \ref{S: IC pairing} when working with intersection chains in order to properly manage the perversities involved.

Also, recalling that two respectively $a$- and $b$-dimensional PL subspaces $A$ and $B$ of an $m$-dimensional PL manifold $M$ are in general position if $\dim(A\cap B)\leq a+b-m$,
we note that the dimension $i-n$ in Definition \ref{D: Delta} is precisely the intersecton dimension we would expect if we had  $i$- and $n$-dimensional subspaces (in this case $|\xi|$ and the diagonal $\Delta(X)$) in general position in a $2n$-dimensional manifold. 
\end{remark}

So let $\xi\in C_i^{\Delta,\infty}((X,\Sigma)\times (X,\Sigma)))$. To define the map $\Delta_!$, we need to say what $\Delta_!(\xi)\in  C^\infty_{i-n}(X,\Sigma)$ is. Due to our focus on relative chains, we only care about simplices  that are not contained in the singular locus $\Sigma$, so to define $\Delta_!(\xi)$ in a given triangulation $T$ of $X$, it suffices to prescribe the coefficient of $\Delta_!(\xi)$ for each simplex of $T$ not contained in $\Sigma$. More precisely, given a sufficiently refined $T$, we define a simplicial chain $\Delta^T_!(\xi)$. 
  We then show that our definition is independent of the choice of triangulation in order to 
  obtain a well-defined  PL chain $\Delta_!(\xi)$. 
  More explicitly, we show that if $T'$ is a subdivision of $T$, then the image of $\Delta^T_!(\xi)$ under the subdivision map is $\Delta^{T'}_!(\xi)$. As any two triangulations of $X$ have a common subdivision (apply \cite[Theorem 3.6.C]{HUD} to the identity map), this suffices to determine a PL chain $\Delta_!(\xi)$.  
  
In fact, we will use triangulations of $X$ of the following form: Suppose $\mc T$ is some triangulation of $X\times X$ 
compatible with the stratification (letting $(X\times X)^i=\cup_{j+k=i}X^j\times X^k$)  and with respect to which $|\xi|$ and $\Delta(X)$ are subcomplexes (apply \cite[Theorem 3.6.C]{HUD} inductively to obtain such a triangulation). Let $T$ be the induced triangulation on $X$ corresponding to the subcomplex triangulation of $\Delta(X)$. Note that using only such triangulations on $X$ is not a strong restriction, as any arbitrary triangulations $T$ of $X$ and $\mc T$ of  $X\times X$ have subdivisions $T'$ and $\mc T'$ such that $T'$ is a subcomplex of $\mc T'$; this follows from \cite[Theorem 3.6.C]{HUD}, as the diagonal inclusion is proper.

So now to define $\Delta^T_!(\xi)$, let $\sigma$ be an $i-n$ simplex of $T$ with  $\sigma\not\subset \Sigma$.
 We will define an \emph{intersection  coefficient} $I_\sigma(\xi)\in \Z$, which ultimately will not depend on the particular choice of triangulation so long as $\sigma$ is a simplex of the triangulation and the triangulation is obtained as above. Then we will define $\Delta^T_!(\xi)= \sum_{\sigma\in T, \sigma\not\subset \Sigma} I_\sigma(\xi)\sigma$, the sum over $i-n$ simplices.  Here, the referenced ``intersection'' justifying the name ``intersection coefficient'' is the intersection of $\xi$ with the diagonal $\Delta(X)$.

We carry out this program in the next two subsections.

\subsection{Local umkehr maps and intersection coefficients}\label{S: local umkehr}

To motivate the definition that follows, recall that the classical diagonal transfer map for a compact oriented $n$-manifold $M$ would have the form $\Delta_!:H_i(M\times M)\to H_{i-n}(M)$ defined by the composition 
$$H_i(M\times M)\xr{D_{M\times M}^{-1}}H^{2n-i}(M\times M)\xr{\Delta^*}H^{2n-i}(M)\xr{D_M}H_{i-n}(M),$$
with $D_{M\times M}$ and $D_M$ being Poincar\'e duality maps (see, for example, \cite[Definition 6.11.2]{BRTG}). The map $\Delta_!$ we seek to define will have the same general character, but as we do not work with manifolds and as our input homology groups are not those of the whole space but rather those representing PL chains, we will need to use a version of the Goresky-MacPherson duality map $\D$ instead of $D$. Furthermore, as we do not work on compact spaces, we begin instead with a local construction that determines intersection coefficients by restricting to compact neighborhoods of the simplex of interest $\sigma$. We will eventually obtain a global map $\Delta_!$ by piecing together this local data.

So to define the intersection coefficient $I^T_\sigma(\xi)$, let $T$, $\xi$, and $\sigma$ all be given as above.
Let $Z$ be any finite union of $n$-simplices of $T$ such that the interior of $\sigma$ is contained in the interior of $Z$; this will be our compact neighborhood of $\sigma$. Such a $Z$ exists; for example we could choose $Z$ to be the union of all $n$-simplices of $T$ containing $\sigma$ as a face. For a chain $\zeta$ in $X$, let $|\zeta|_Z=|\zeta|\cap Z$, and, similarly, for a chain $\zeta$ in $X\times X$, let $|\zeta|_{Z\times Z}=|\zeta|\cap (Z\times Z)$. Let $\Sigma_Z=\Sigma\cap Z$. These notations will be useful for restricting objects to the neighborhoods $Z$ or $Z\times Z$.

Let $D_Z$ be the union of all $n$-simplices of $T$ not contained in $Z$, and let $S_Z=Z\cap D_Z$.
So $D_Z$ is morally the complement of $Z$ and $S_Z$ is the boundary between $Z$ and $D_Z$.
We also let $D_{Z\times Z}=(X\times D_Z)\cup (D_Z\times X)$.  Let $J_Z=\Sigma_Z\cup S_Z$, which we will see is essentially the complement of the manifold points in $Z$, and let $J_{Z\times Z}=(Z\times J_Z)\cup(J_Z\times Z)$, which will play the same role for $Z\times Z$. Note that $\Delta^{-1}(J_{Z\times Z})=J_Z$. We also recall that if $A\in X\times X$, then we let $A'=\Delta^{-1}(A)\cong A\cap \Delta(X)$. 

 We will show below that 
 $Z-J_Z$ and $Z\times Z-J_{Z\times Z}$ are indeed oriented $n$-dimensional manifolds, with the orientations inherited from $X$ and $X\times X$, respectively. Thus $(Z,J_Z)$ and $(Z\times Z,J_{Z\times Z})$ each satisfy the first hypothesis of Proposition \ref{P: GM dual}.

\begin{definition}\label{D: int coeff Z}
Given the assumptions above, we define the  \emph{$(Z,T)$-intersection coefficient} $I^{Z,T}_\sigma(\xi)$ as follows for an $i-n$ simplex $\sigma$, $\sigma \not\subset \Sigma$, of the triangulation $T$: 
If $\sigma\not\subset \Delta^{-1}(|\xi|)=|\xi|'$, then define  $I_{\sigma}^{Z,T}(\xi)$ to be $0$. Otherwise,  let $[\xi]\in H_i^{\infty}(|\xi|,|\bd \xi|)$ represent $\xi$  under the isomorphism $\alpha^{-1}_{|\xi|,|\bd \xi|}$ of Lemma \ref{L: useful} and  let $I^{Z,T}_\sigma(\xi)$   be the image of $[\xi]$ under the following composition:

\begin{align*}
H_i^{\infty}(|\xi|,|\bd \xi|)
&\to H_i^{\infty}(|\xi|\cup J_{Z\times Z}\cup D_{Z\times Z},|\bd \xi|\cup J_{Z\times Z}\cup D_{Z\times Z})\\
&\xleftarrow{\cong}H_i(|\xi|_{Z\times Z}\cup J_{Z\times Z}, |\bd \xi|_{Z\times Z}\cup J_{Z\times Z})\\
&\xleftarrow{\D} H^{2n-i}(Z\times Z-(|\bd \xi|_{Z\times Z}\cup J_{Z\times Z}),Z\times Z-(|\xi|_{Z\times Z}\cup J_{Z\times Z}))\\
&\xr{\Delta^*}  H^{2n-i}(Z-|\bd \xi|_Z'\cup J_Z,Z-(|\xi|_Z'\cup J_Z))\\
&\xr{\D} H_{i-n}(|\xi|_Z'\cup J_Z,|\bd \xi|_Z'\cup J_Z)\\
&\to H_{i-n}((|\xi|'_Z\cup J_Z,(|\xi|_Z'\cup J_Z)-int(\sigma))\\
&\xleftarrow{\cong}  H_{i-n}(\sigma, \bd \sigma)\\
&\cong \Z.
\end{align*}
The last isomorphism is determined by the orientation of $\sigma$. 

We will show below that this construction does not depend on the choice of $Z$ or the triangulation $T$, given the previous constraints, and so we can define the \emph{intersection coefficient} $I_{\sigma}(\xi)$ to be $I^{Z,T}_{\sigma}(\xi)$ for any choice of $Z$ or of $T$ containing $\sigma$. 
\end{definition}

\paragraph{Well-definedness.}
Let us first observe that $I^{Z,T}_{\sigma}(\xi,\eta)$  is  well defined: 
\begin{enumerate}
\item The first map is induced by inclusions. 

\item The second map is  an excision isomorphism.
To see that this is a valid excision, we note that 
\begin{align*}
(|\bd \xi|\cup J_{Z\times Z}\cup D_{Z\times Z})&\cap (|\xi|_{Z\times Z}\cup J_{Z\times Z})\\
&=|\bd \xi|_{Z\times Z}\cup J_{Z\times Z},
\end{align*}
as $D_{Z\times Z}\cap |\xi|_{Z\times Z}\subset D_{Z\times Z}\cap (Z\times Z)=(Z\times S_Z)\cup (S_Z\times Z)\subset J_{Z\times Z}$. So the excision removes $D_{Z\times Z}-((Z\times Z)\cap D_{Z\times Z})$.
As $Z\times Z$ can be triangulated as a finite subcomplex of $X\times X$, the argument for excision is identical to the standard simplicial homology excision proof, e.g.\ \cite[Theorem 9.1]{MK}. 
Finally, as $|\xi|_{Z\times Z}\cup J_{Z\times Z}$ and $|\bd \xi|_{Z\times Z}\cup J_{Z\times Z}$ are compact, in particular they will be finite subcomplexes in any triangulation of $Z\times Z$, we have $H^\infty_i(|\xi|_{Z\times Z}\cup J_{Z\times Z}, |\bd \xi|_{Z\times Z}\cup J_{Z\times Z})=H_i(|\xi|_{Z\times Z}\cup J_{Z\times Z}, |\bd \xi|_{Z\times Z}\cup J_{Z\times Z})$, which can from here on also be identified with the singular homology group. 

\item The third map and fifth maps are the isomorphisms of Proposition \ref{P: GM dual}, once we show that $Z-J_Z$ and $Z\times Z-J_{Z\times Z}$ are manifolds. 

\item The fourth map is the pullback by the diagonal map. 

\item\label{I: gp sigma} The sixth map is induced by inclusion, noting that, by our assumptions, $\sigma$ is contained in $|\xi|'_Z$ and the interior of $\sigma$ is not contained in $|\bd \xi|'_Z\cup J_Z$. In fact, by assumption of general position, $\dim(|\xi|'_Z)\leq i-n$ and $\dim(|\bd\xi|'_Z)\leq i-n-1$. So, in particular, if $\sigma$ is an $i-n$ simplex contained in $|\xi|'_Z$, then $\sigma$ cannot be contained in $|\bd\xi|'_Z$. It is also outside of $S_Z$ and $\Sigma_Z$ by assumption.

\item The seventh map is again an excision. 

\item The last isomorphism is standard and determined by the orientation of $\sigma$. 
\end{enumerate}

To finish showing that the intersection coefficient $I^{Z,T}_\sigma(\xi)$ is well defined, we must show that $Z-J_Z$ and $Z\times Z-J_{Z\times Z}$ are manifolds, as claimed. We do this now:

\begin{lemma}
$Z-J_Z$ and $Z\times Z-J_{Z\times Z}$  are oriented manifolds with orientations inherited respectively from $X$ and $X\times X$.
\end{lemma}
\begin{proof}
We begin with $Z-J_Z=Z-(\Sigma_Z\cup S_Z)$. Once we have shown that it is a manifold, the orientability will follow because then $Z-(\Sigma_Z\cup S_Z)$ must be a submanifold of $X-\Sigma_X$, which is oriented.

Let $x\in Z-(\Sigma_Z\cup S_Z)$. We must show that $x$ has an $n$-dimensional Euclidean neighborhood. Clearly this is true at any point that is interior to an $n$-simplex of $Z$. So suppose that $x$ is contained in the interior of a simplex $\tau$ of $Z$ in the triangulation $T$ with respect to which $Z$ is defined and that $\dim(\tau)<n$. Thinking of $x$ as a point of $X$, we have $x\in X-\Sigma_X$ and so $x$ has a Euclidean neighborhood, say $U$, in $X$, and we may suppose $U$ is contained in a star neighborhood of $\tau$ in $X$ (in the triangulation $T$). If the star neighborhood of $\tau$ is contained in $Z$, then $U$ is a Euclidean neighborhood of  $x$ in $Z$. If not, then there is a simplex $\delta$ of $X$ that contains $\tau$ as a face and that is not contained in $Z$.  But, as $X$ is a pseudomanifold, every simplex of $X$ is contained in some $n$-simplex, and this implies there is some $n$-simplex $\gamma$, having $\delta$ and $\tau$ as faces and such that $\gamma$ is not contained in $Z$ (if any $n$-simplex having $\delta$ as a face were contained in $Z$, then $\delta$ would be contained in $Z$). So then $\tau$ is contained in $Z$ and in $D_Z$, which was defined as the union of $n$-simplices not contained in $Z$. So $\{x\}\subset \tau\subset S_Z$, a contradiction. 

Thus all points  $x\in Z-(\Sigma_Z\cup S_Z)$ have $n$-dimensional Euclidean neighborhoods, and $Z-(\Sigma_Z\cup S_Z)$ is a subspace of a Hausdorff space, so it is a manifold.

For $Z\times Z-J_{Z\times Z}$, we recall that $J_{Z\times Z}=(Z\times J_Z)\cup(J_Z\times Z)$ so $Z\times Z-J_{Z\times Z}=(Z-J_Z)\times (Z-J_Z)$, and the claim follows. 
\end{proof}

\begin{remark}\label{R: subdiv}
Note that while the triangulation $T$ is utilized in our selection of $\sigma$ and $Z$, once this simplex and subspace have been established,  nothing in the chain of maps defining $I^{Z,T}_{\sigma}(\xi)$ depends on the specific triangulation. Thus, we can define $I^{Z}_{\sigma}(\xi)=I^{Z,T}_{\sigma}(\xi)$ for any suitable triangulation. Furthermore, we see that if we replace $T$ with some subdivision $T'$, then $Z$ and  the other spaces involved in defining our intersection  coefficient are still well-defined as simplicial subspaces of $X$ in this subdivision. Furthermore, if $\tau$ is an $i-n$ simplex of $T'$ with $\tau\subset \sigma$ and with $\tau$ oriented consistently with $\sigma$, then the diagram 

\begin{diagram}
H_{i-n}(|\xi|_Z'\cup J_Z,|\bd \xi|_Z'\cup J_Z)\\
\dTo&\rdTo(4,2)\\
 H_{*}((|\xi|'_Z\cup J_Z,(|\xi|_Z'\cup J_Z)-int(\sigma))&&\rTo&&H_{*}((|\xi|'_Z\cup J_Z,(|\xi|_Z'\cup J_Z)-int(\tau)))\\
\uTo^\cong&&&\ruTo^\cong&\uTo^\cong\\
 H_{*}(\sigma, \bd \sigma)&\rTo&H_*(\sigma,\sigma-int(\tau))&\lTo^\cong&H_{*}(\tau, \bd \tau))
 \end{diagram}
 shows that $I^{Z}_{\tau}(\xi)=I^{Z, T'}_{\tau}(\xi)$ agrees with  $I^{Z}_{\sigma}(\xi)=I^{Z,T}_{\sigma}(\xi)$.
\end{remark}

\begin{remark}\label{R: patch}
 By assumption, $\dim(|\xi|'_Z)\leq i-n$, and therefore, by Lemma \ref{L: useful coeff big} (taking $B=C$ there), an element of $H_{i-n}(|\xi|_Z'\cup J_Z,|\bd \xi|_Z'\cup J_Z)$, as occurs in the middle of our definition of $I^{Z,T}_\sigma(\xi)$, determines a PL chain in $Z$ that consists, in any compatible triangulation, of $i-n$ simplices not contained in $\Sigma_Z$ or $S_Z$ and whose boundary is contained in $|\bd \xi|'_Z\cup \Sigma_Z\cup S_Z$. 
We can think of this chain as the \emph{local umkehr image of $\xi$ in $Z$}, which we denote $\Delta_!^{Z}(\xi)$. If we wish to fix a given triangulation, we can refer more specifically to the simplicial chain 
$\Delta_!^{Z,T}(\xi)$ whose underlying PL chain is $\Delta_!^{Z}(\xi)$.
 The further refinement in the definition to $I^{Z,T}_\sigma(\xi)$, itself, is then just an example of the process  described in Remarks \ref{R: useful coeff} and \ref{R: useful coeff big} for finding the coefficients of simplices in a simplicial representation of a PL chain. Our construction of $\Delta_!(\xi)$, which is  described in the next section, can then be thought of alternatively as patching together these local images to obtain a single global chain in $X$ with boundary in $|\bd \xi|\cup \Sigma$. The independence of choice of $Z$ that we are about to demonstrate implies that this patching is well defined locally, and so determines a unique such global image chain. 
\end{remark}

\paragraph{Independence of $Z$.}
Next we want to show that $I^Z_{\sigma}(\xi)$ is independent of the choice of $Z$. It is sufficient to demonstrate this property with respect to an alternative choice, say $Y$, in any triangulation compatible with the stratification and containing $|\xi|$ and $\Delta(X)$ as subcomplexes,   and such that $\sigma \subset Y\subset Z$. For if $Z'$ is an arbitrary alternative to $Z$ then we can always find a $Y$ such that $Z\supset Y\subset Z'$ by taking $Y$ to be a star neighborhood of $\sigma$ in a sufficiently fine triangulation. Then the independence of choices between $Z$ and $Y$ and between $Y$ and $Z'$ demonstrate the overall independence of choice of $Z$.

\begin{proposition}\label{P: indep Z}
The intersection coefficient $I^Z_{\sigma}(\xi)$ is independent of the choice of $Z$. 
\end{proposition}
\begin{proof}
The proof will proceed through a series of commutative diagrams, throughout which the top lines will compose to the definition of $I^{Z,T}_\sigma(\xi)$ and the bottom lines will compose to the definition of $I^{Y,T}_\sigma(\xi)$. The middle line will be an intermediary that accepts maps from the top and bottom lines, but the middle and lower lines will be isomorphic throughout, allowing us to reverse the lower vertical maps to obtain a direct comparison between $I^{Z,T}_\sigma(\xi)$ and $I^{Y,T}_\sigma(\xi)$.

The following observations will be useful: As $Y\subset Z$, it follows that
$D_Z\subset D_Y$, $S_Z\subset D_Y$, $D_{Z\times Z}\subset D_{Y\times Y}$,  if $A$ is any subset of $X$ then $(A_Z)_Y=A_Y$, and if $A$ is any subset of $X\times X$ then $(A_{Z\times Z})_{Y\times Y}=A_{Y\times Y}$. Also, $J_{Y\times Y}\subset J_{Z\times Z}\cup D_{Y\times Y}$, and, in fact, one can check that $J_{Y\times Y}\cup D_{Y\times Y}=J_{Z\times Z}\cup D_{Y\times Y}$, using that $\Sigma_Z\subset \Sigma_Y\cup D_Y$. Similarly, $J_Z\cup D_Y=J_Y\cup D_Y$.

First, we have the following:

\medskip

{\resizebox{.95\hsize}{!}{
\begin{diagram}
&&H_i^{\infty}(|\xi|\cup J_{Z\times Z}\cup D_{Z\times Z},|\bd \xi|\cup J_{Z\times Z}\cup D_{Z\times Z})&\lTo^{\cong}&H_i(|\xi|_{Z\times Z}\cup J_{Z\times Z}, |\bd \xi|_{Z\times Z}\cup J_{Z\times Z})\\
&\ruTo&\dTo&&\dTo\\
H_i^{\infty}(|\xi|,|\bd \xi|)&\rTo& H_i^{\infty}(|\xi|\cup J_{Z\times Z}\cup D_{Y\times Y},|\bd \xi|\cup J_{Z\times Z}\cup D_{Y\times Y})&\lTo^{\cong}&H_i(|\xi|_{Z\times Z}\cup J_{Z\times Z}\cup(D_{Y\times Y}\cap (Z\times Z)), |\bd \xi|_{Y\times Y}\cup J_{Y\times Y}\cup(D_{Y\times Y}\cap (Z\times Z)))\\
&\rdTo&\uTo^=&&\uTo^\cong\\
&& H_i^{\infty}(|\xi|\cup J_{Y\times Y}\cup D_{Y\times Y},|\bd \xi|\cup J_{Y\times Y}\cup D_{Y\times Y})&\lTo^{\cong}&H_i(|\xi|_{Y\times Y}\cup J_{Y\times Y}, |\bd \xi|_{Y\times Y}\cup J_{Y\times Y}).\\
\end{diagram}}}

\bigskip

The commutativity is clear at the space level as all maps are induced by spatial inclusions, using the observations concerning spatial relations given above. Notice also that   $|\bd \xi|_{Z\times Z}$ is the union of $|\bd \xi|_{Y\times Y}$ with the simplices of $|\bd \xi|$ in $D_{Y\times Y}\cap(Z\times Z)$. The right horizontal arrow in the middle row excises $(X\times (D_Z-S_Z))\cup ((D_Z-S_Z)\times X)$, and so is an isomorphism. Thus the bottom right vertical map is also an isomorphism.  


Next, we have a diagram 

\bigskip

\resizebox{.95\hsize}{!}{
\begin{diagram}
H_i(|\xi|_{Z\times Z}\cup J_{Z\times Z}, |\bd \xi|_{Z\times Z}\cup J_{Z\times Z})&\lTo^{\D}&H^{2n-i}(Z\times Z-(|\bd \xi|_{Z\times Z}\cup J_{Z\times Z}),Z\times Z-(|\xi|_{Z\times Z}\cup J_{Z\times Z}))\\
\dTo&&\dTo\\
H_i(|\xi|_{Z\times Z}\cup J_{Z\times Z}\cup(D_{Y\times Y}\cap (Z\times Z)), |\bd \xi|_{Y\times Y}\cup J_{Y\times Y}\cup(D_{Y\times Y}\cap (Z\times Z)))&\lTo^{\D}&H^{2n-i}(Z\times Z-(|\bd \xi|_{Y\times Y}\cup J_{Y\times Y}\cup(D_{Y\times Y}\cap (Z\times Z))),Z\times Z-(|\xi|_{Z\times Z}\cup J_{Z\times Z}\cup(D_{Y\times Y}\cap (Z\times Z)))\\
\uTo^\cong&&\uTo^=\\
H_i(|\xi|_{Y\times Y}\cup J_{Y\times Y}, |\bd \xi|_{Y\times Y}\cup J_{Y\times Y})&\lTo^{\D}&H^{2n-i}(Y\times Y-(|\bd \xi|_{Y\times Y}\cup J_{Y\times Y}),Y\times Y-(|\xi|_{Y\times Y}\cup J_{Y\times Y})).
\\
\end{diagram}}

\bigskip

The top commutes by Lemma \ref{L: natural dual}, using $(Z\times Z,J_{Z\times Z})$ for the pair ``$(Z,S)$''. 
By Lemma \ref{L: sing indep}, we can also treat the middle duality isomorphism as utilizing $(Z\times Z,J_{Y\times Y}\cup(D_{Y\times Y}\cap (Z\times Z)))$ for the pair ``$(Z,S)$.'' Then the bottom square commutes by Lemma \ref{L: expansion}, letting the $S$ of that lemma be $J_{Y\times Y}$ and the $T$ of that lemma be $D_{Y\times Y}\cap (Z\times Z)$.

We then have the naturality of the pullback:

\bigskip

\resizebox{.95\hsize}{!}{
\begin{diagram}
H^{2n-i}(Z\times Z-(|\bd \xi|_{Z\times Z}\cup J_{Z\times Z}),Z\times Z-(|\xi|_{Z\times Z}\cup J_{Z\times Z}))&\rTo^{\Delta^*}&H^{2n-i}(Z-|\bd \xi|_Z'\cup J_Z,Z-(|\xi|_Z'\cup J_Z))\\
\dTo&&\dTo\\
H^{2n-i}(Z\times Z-(|\bd \xi|_{Y\times Y}\cup J_{Y\times Y}\cup(D_{Y\times Y}\cap (Z\times Z))),Z\times Z-(|\xi|_{Z\times Z}\cup J_{Z\times Z}\cup(D_{Y\times Y}\cap (Z\times Z))))&\rTo^{\Delta^*}&H^{2n-i}(Z-|\bd \xi|_Y'\cup J_Y\cup(D_Y\cap Z),Z-(|\xi|_Z'\cup J_Z\cup(D_Y\cap Z)))\\
\uTo^=&&\uTo^=\\
H^{2n-i}(Y\times Y-(|\bd \xi|_{Y\times Y}\cup J_{Y\times Y}),Y\times Y-(|\xi|_{Y\times Y}\cup J_{Y\times Y}))&\rTo^{\Delta^*}&H^{2n-i}(Y-|\bd \xi|_Y'\cup J_Y,Y-(|\xi|_Y'\cup J_Y)).
\\
\end{diagram}}

\bigskip

This is followed again by a commutative duality map

\bigskip

\resizebox{.95\hsize}{!}{
\begin{diagram}
H^{2n-i}(Z-|\bd \xi|_Z'\cup J_Z,Z-(|\xi|_Z'\cup J_Z))&\rTo^{\D}&H_{i-n}(|\xi|_Z'\cup J_Z,|\bd \xi|_Z'\cup J_Z)\\
\dTo&&\dTo\\
H^{2n-i}(Z-|\bd \xi|_Y'\cup J_Y\cup(D_Y\cap Z),Z-(|\xi|_Z'\cup J_Z\cup(D_Y\cap Z)))&\rTo^{\D}&H_{i-n}(|\xi|_Z'\cup J_Y\cup(D_Y\cap Z),|\bd \xi|_Z'\cup J_Z\cup (D_Y\cap Z))\\
\uTo^=&&\uTo^\cong\\
H^{2n-i}(Y-|\bd \xi|_Y'\cup J_Y,Y-(|\xi|_Y'\cup J_Y))&\rTo^{\D}&H_{i-n}(|\xi|_Y'\cup J_Y,|\bd \xi|_Y'\cup J_Y).
\\
\end{diagram}}

\bigskip

Again, the top commutes by Lemma \ref{L: natural dual}, using $(Z,J_Z)$ for the pair ``$(Z,S)$''. 
By Lemma \ref{L: sing indep}, we can also treat the middle duality isomorphism as utilizing $(Z,J_Y\cup (D_Y\cap Z))$ for the pair ``$(Z,S)$.'' Then the bottom square commutes by Lemma \ref{L: expansion}, letting the $T$ of that lemma be $D_Y\cap Z$. 

Finally, we have the diagram 

\bigskip

\resizebox{.95\hsize}{!}{
\begin{diagram}
H_{i-n}(|\xi|_Z'\cup J_Z,|\bd \xi|_Z'\cup J_Z))&\rTo&H_{i-n}((|\xi|'_Z\cup J_Z,(|\xi|_Z'\cup J_Z)-int(\sigma))&\lTo^\cong& H_{i+j-n}(\sigma, \bd \sigma)\\
\dTo&&\dTo&&\dTo^=\\
H_{i-n}(|\xi|_Z'\cup J_Y\cup(D_Y\cap Z),|\bd \xi|_Z'\cup J_Z\cup (D_Y\cap Z))&\rTo&H_{i-n}((|\xi|'_Z\cup J_Y\cup(D_Y\cap Z),(|\xi|_Z'\cup J_Y\cup(D_Y\cap Z))-int(\sigma))&\lTo^\cong& H_{i+j-n}(\sigma, \bd \sigma) \\
\uTo^\cong&&\uTo^\cong&&\uTo^\cong^=\\
H_{i-n}(|\xi|_Y'\cup J_Y,|\bd \xi|_Y'\cup J_Y) &\rTo&H_{i-n}((|\xi|'_Y\cup J_Y,(|\xi|_Y'\cup J_Y)-int(\sigma))&\lTo^\cong& H_{i+j-n}(\sigma, \bd \sigma).
\end{diagram}}

\bigskip

The lower vertical map in the middle column is an isomorphism from the commutativity of the diagram and from the other excision isomorphisms.

Putting all these pieces together, we see that $I^{Z,T}_{\sigma}(\xi)=I^{Y,T}_{\sigma}(\xi)$, which is the desired independence result.
\end{proof}

\begin{definition}
Proposition \ref{P: indep Z}, together  with Remark \ref{R: subdiv},  demonstrates that we are justified in defining the \emph{intersection coefficient} $I_{\sigma}(\xi)$ to be  $I^{Z,T}_{\sigma}(\xi)$ for any appropriate choice of $Z$ and $T$, without ambiguity.
\end{definition}

\paragraph{Additivity.}
We would next like to show that if we have $\xi_1,\xi_2\in C_i^{\Delta,\infty}((X,\Sigma)\times (X,\Sigma)))$,  then 
$I_{\sigma}(\xi_1+\xi_2)=I_{\sigma}(\xi_1)+I_{\sigma}(\xi_2)$. Below, this will allow us to see that the intersection product is a homomorphism. This point is somewhat neglected in all the previous treatments \cite{GM1, McC, GBF18}, though it is not completely trivial.  The tricky point is that $I_{\sigma}(\xi_1)$ is defined using a representation of  $\xi_1$ in $H_i^\infty(|\xi_1|,|\bd \xi_1|)$, while $I_{\sigma}(\xi_2)$ is defined using a representation of $\xi_2$ in $H_i^\infty(|\xi_2|,|\bd \xi_2|)$, and $I_{\sigma}(\xi_1+\xi_2)$ is defined using a representation of $\xi_1+\xi_2$ in $H_i^\infty(|\xi_1+\xi_2|,|\bd (\xi_1+\xi_2)|)$. Beyond needing to find a way to relate these groups, we also have the issue that $|\xi_1+\xi_2|\neq |\xi_1|\cup|\xi_2|$ in general, as there may be cancellation of simplices in the sum. The solution is to observe that the use of $H_i^\infty(|\xi|,|\bd \xi|)$ in the first step of the definition of $I_\sigma(\xi)$ is overly restrictive. This is the content of the next lemma.

\begin{lemma}\label{L: add}
Let $B\subset A\subset X\times X$ be PL subspaces such that 
\begin{enumerate}
\item $\dim(A)=i$ and $\dim(B)<i$,
\item $|\xi|\subset A$ and $|\bd \xi|\subset B$
\item $A$ and $B$ are in  general position with respect to $\Delta$. 
\end{enumerate}
 Then for $\sigma\subset A\cap \Delta(X)$ and $\sigma\not\subset \Sigma$, the definition of $I_\sigma(\xi)$ remains unchanged replacing $[\xi]\in H_i^\infty(|\xi|,|\bd \xi|)$ with its image in $H_i^\infty(A,B)$ in the first step of the construction of Definition \ref{D: int coeff Z} and then replacing $(|\xi|,|\bd \xi|)$ with $(A,B)$ throughout. If $\sigma\not\subset A\cap \Delta(X)$, we set $I_\sigma(\xi)=0$, which is consistent with the original definition.
\end{lemma}

\begin{proof}
Definition \ref{D: int coeff Z} defines  $I_\sigma(\xi)$ in terms of $I^{Z,T}_\sigma(\xi,\eta)$. We show here that $I^{Z,T}_\sigma(\xi)$ does not change when we replace $(|\xi|,|\bd \xi|)$ by $(A,B)$, but the independence of $I^{Z,T}_\sigma(\xi)$ with respect to choices of $Z$ and $T$ will then indicate that the same is true using $(A,B)$.

Choose $T$, $\sigma$, and $Z$ as in Definition \ref{D: int coeff Z}, though assuming also that $T$ is compatible with $A$ and $B$. We first suppose $\sigma\subset |\xi|_Z'$ and $\sigma\not\subset \Sigma_X$. Let $A'=\Delta^{-1}(A)=A\cap \Delta(X)$ and define  $B'$ similarly. Then we have the following  diagram

{\scriptsize
\begin{diagram}
H_i^{\infty}(|\xi|,|\bd \xi|)&\rTo&H_i^{\infty}(A,B)\\
\dTo&&\dTo\\
H_i^{\infty}(|\xi|\cup J_{Z\times Z}\cup D_{Z\times Z},|\bd \xi|\cup J_{Z\times Z}\cup D_{Z\times Z})&\rTo&H_i^{\infty}(A\cup J_{Z\times Z}\cup D_{Z\times Z},B\cup J_{Z\times Z}\cup D_{Z\times Z})\\
\uTo^\cong&&\uTo^\cong\\
H_i(|\xi|_{Z\times Z}\cup J_{Z\times Z}, |\bd \xi|_{Z\times Z}\cup J_{Z\times Z})&\rTo&H_i(A_{Z\times Z}\cup J_{Z\times Z}, B_{Z\times Z}\cup J_{Z\times Z})\\
\uTo^{\D}&&\uTo^{\D}\\
 H^{2n-i}(Z\times Z-(|\bd \xi|_{Z\times Z}\cup J_{Z\times Z}),Z\times Z-(|\xi|_{Z\times Z}\cup J_{Z\times Z}))&\rTo&H^{2n-i}(Z\times Z-(B_{Z\times Z}\cup J_{Z\times Z}),Z\times Z-(A_{Z\times Z}\cup J_{Z\times Z}))\\
\dTo^{\Delta^*} &&\dTo^{\Delta^*} \\
 H^{2n-i}(Z-|\bd \xi|_Z'\cup J_Z,Z-(|\xi|_Z'\cup J_Z))&\rTo&H^{2n-i}(Z-B_Z'\cup J_Z,Z-(A_Z'\cup J_Z))\\
\dTo^{\D}&&\dTo^{\D}\\
H_{i-n}(|\xi|_Z'\cup J_Z,|\bd \xi|_Z'\cup J_Z)&\rTo&H_{i-n}(A_Z'\cup J_Z,B_Z'\cup J_Z)\\
\dTo&&\dTo\\
H_{i-n}((|\xi|'_Z\cup J_Z,(|\xi|_Z'\cup J_Z)-int(\sigma)&\rTo&H_{i-n}((A'_Z\cup J_Z,(A_Z'\cup J_Z)-int(\sigma)\\
\uTo^{\cong} &&\uTo^{\cong}\\
H_{i-n}(\sigma, \bd \sigma)&\rTo&H_{i-n}(\sigma, \bd \sigma)\\
\uTo^\cong&&\uTo^\cong\\
\Z&\rTo&\Z.
\end{diagram}}

The diagram commutes by the obvious space inclusions and by Lemma \ref{L: natural dual}, for the duality maps; we can use $(Z,J_Z)$ and $(Z\times Z,J_{Z\times Z})$ as the pairs ``(X,S)'' for the duality isomorphisms. 
Note that the inclusion map $H_{i-n}(A_Z'\cup J_Z,B_Z'\cup J_Z)\to H_{i-n}(A'_Z\cup J_Z,(A_Z'\cup J_Z)-int(\sigma))$ utilizes our usual assumption that $\sigma$ is not contained in $J_Z$, while the assumptions on $B$ assure that $\dim(B_X')<i-n$. 
Similarly, the excision isomorphism $H_{i-n}(\sigma, \bd \sigma)\to H_{i-n}((A'_Z\cup J_Z,(A_Z'\cup J_Z)-int(\sigma))$ uses the dimension assumptions to assure that $\sigma\cap ((A_Z'\cup J_Z)-int(\sigma))=\bd \sigma$. 
The diagram shows that, in this case, $I^{Z,T}_{\sigma}(\xi)$ can indeed be computing utilizing $(A,B)$  as claimed.

Next, suppose $\sigma\not\subset A'$. Then $\sigma\not\subset |\xi|'$, so $I^{Z,T}_\sigma(\xi)=0$ by definition, which is consistent with the statement of the lemma here. 

Finally, suppose $\sigma\subset A'$ but $\sigma\not\subset |\xi|'$. Then $I^{Z,T}_\sigma(\xi)=0$ by definition, and we must show that the image down the right side of our diagram of $[\xi]\in H_i^{\infty}(A,B)$   is $0$. But notice that the top part of the diagram does not depend on $\sigma$, so the image of $[\xi]$ in   $H_{i-n}(A_Z'\cup J_Z,B_Z'\cup J_Z)$ must be in the image of $H_{i-n}(|\xi|_Z'\cup J_Z,|\bd \xi|_Z'\cup J_Z)$. Therefore, it can be represented by a chain that does not include $\sigma$ and so represents $0$ in 
$H_{i-n}(A'_Z\cup J_Z,(A_Z'\cup J_Z)-int(\sigma))$, as desired.
\end{proof}

\begin{proposition}\label{P: additivity}
Suppose  $\xi_1,\xi_2\in C^{\Delta,\infty}_i((X,\Sigma)\times (X,\Sigma)))$. Then 
$I_{\sigma}(\xi_1+\xi_2)=I_{\sigma}(\xi_1)+I_{\sigma}(\xi_2)$. 
\end{proposition}
\begin{proof}
By Lemma \ref{L: add}, $I_{\sigma}(\xi_1+\xi_2)$, $I_{\sigma}(\xi_1)$, and $I_{\sigma}(\xi_2)$ can all be computed using the formula of Definition \ref{D: int coeff Z} beginning with representations of the chains $\xi_1$, $\xi_2$, and $\xi_1+\xi_2$ as elements $[\xi_1], [\xi_2],[\xi_1+\xi_2]\in H_i^\infty(|\xi_1|\cup|\xi_2|,|\bd \xi_1|\cup |\bd \xi_2|)$. Clearly  $[\xi_1]+[\xi_2]=[\xi_1+\xi_2]$, using the isomorphism   $\alpha_{|\xi_1|\cup|\xi_2|,|\bd \xi_1|\cup |\bd \xi_2|}$ of Lemma \ref{L: useful}. From here, the result follows using that the maps of Definition \ref{D: int coeff Z} are all homomorphisms. 
\end{proof}

\subsection{The umkehr map}\label{S: umkehr}
We can now officially define the umkehr map $\Delta_!:C^{\Delta,\infty}_*((X,\Sigma)\times (X,\Sigma))\to C^{\infty}_{*-n}(X,\Sigma)$ as follows:

\begin{definition}
Let $X$ be an oriented $n$-dimensional PL stratified pseudomanifold, not necessarily compact, and let $\xi\in C^{\Delta,\infty}_i((X,\Sigma)\times (X,\Sigma))$.   If $T$ is any triangulation of  $X$ restricting a triangulation of $X\times X$ that is compatible with the stratification and for which  $|\xi|$ and $\Delta(X)$ are subcomplexes,  then
define $\Delta^T_!(\xi)\in c^{T,\infty}_{i-n}(X,\Sigma)$
to be $$\Delta^T_!(\xi)=\sum_{\sigma\in T, \sigma\not\subset \Sigma} I_{\sigma}(\xi)\sigma.$$
By Remark \ref{R: subdiv}, the image of $\Delta^T_!(\xi)$ in $C^{\infty}_{i-n}(X,\Sigma)$ does not depend on $T$, and we  let this image be $\Delta_!(\xi)$.  It follows from Proposition \ref{P: additivity} that $\Delta_!$ is a homomorphism $C^{\Delta,\infty}_i((X,\Sigma)\times (X,\Sigma)) \to C^{\infty}_{i-n}(X,\Sigma)$. 
\end{definition}

Next we show that $\Delta_!$ is a chain map of degree $-n$. Recall (\cite[Remark VI.10.5]{Dold})  that this means that $\Delta_!$ lowers the degree by $n$ and that $\bd \Delta_!=(-1)^n\Delta_!\bd$. 

\begin{proposition}\label{P: chain map}
Let $X$ be an oriented $n$-dimensional PL stratified pseudomanifold, not necessarily compact, and let $\xi\in C^{\Delta,\infty}_i((X,\Sigma)\times (X,\Sigma))$. Then $\Delta_!(\bd \xi)=(-1)^n\bd(\Delta_!(\xi))\in C_{i-n-1}(X,\Sigma)$.
\end{proposition}
\begin{proof}

It is sufficient  to demonstrate that this proposition holds with respect to some fixed triangulation $T$ of the type we have been using. Given such a triangulation, we next observe that all boundary computations are local. In other words, if $\tau$ is an $i-n-1$ simplex of $T$, not contained in $\Sigma$, then the coefficient of $\tau$ in $\bd(\Delta_!(\xi))$ depends only on the coefficients of the simplices of $\Delta_!(\xi)$ that contain $\tau$ as a face. Thus if $Z$ is a neighborhood of $\tau$ containing all the $i-n$ simplices adjacent to $\tau$, we can compute the coefficient of $\tau$ in $\bd(\Delta_!(\xi))$ using only $\Delta^{Z,T}_!(\xi)$ (see Remark \ref{R: patch}). Let us first use this observation to show that both $\Delta_!(\bd \xi)$ and $\bd(\Delta_!(\xi))$ are only made up of simplices in $|\bd \xi|'$. This is clear for $\Delta_!(\bd \xi)$ by the construction of $\Delta_!$. For $\bd( \Delta_!(\xi))$, we know from Remark \ref{R: patch} that each  $\Delta^{Z,T}_!(\xi)$ is a chain composed of $i-n$ simplices not contained in $\Sigma_Z$ or $S_Z$ and whose boundary is contained in $|\bd \xi|'_Z\cup \Sigma_Z\cup S_Z$.
Since we work with relative chains, we are not interested in coefficients of simplices in $\Sigma$, while the spaces $S_Z$ are artifacts of the choice of $Z$. Consequently, as we patch the various $\Delta^{Z,T}_!(\xi)$ together to form $\Delta_!(\xi)$, we can get a better look at the simplices in $\bd (\Delta_!(\xi))$ that might possibly lie in $S_Z$ by choosing a more appropriate $Z$. Ultimately, by ``moving our microscope,'' we see that all of  $\bd (\Delta_!(\xi))$ must be contained in $|\bd \xi|'$ (or $\Sigma$). 

Thus, to prove the proposition, it suffices to show for each $\tau\in |\bd \xi|'$ with $\tau\not\subset \Sigma$ that $\Delta^{Z,T}_!(\bd \xi)$ and $(-1)^n\bd(\Delta^{Z,T}_!(\xi))$ agree for a large enough choice of $Z$, say one containing the star neighborhoods of all $i-n$ simplices having $\tau$ as a face. 
For such a $\tau$, we can now utilize the following  diagram:

{\tiny
\begin{diagram}
H_i^{\infty}(|\xi|,|\bd \xi|)&\rTo^{\bd_*}&H_{i-1}^{\infty}(|\bd \xi|)\\
\dTo&&\dTo\\
H_i^{\infty}(|\xi|\cup J_{Z\times Z}\cup D_{Z\times Z},|\bd \xi|\cup J_{Z\times Z}\cup D_{Z\times Z})&\rTo^{\bd^*}&H_{i-1}^{\infty}(|\bd \xi|\cup J_{Z\times Z}\cup D_{Z\times Z},J_{Z\times Z}\cup D_{Z\times Z})\\
\uTo^\cong&&\uTo^\cong\\
H_i(|\xi|_{Z\times Z}\cup J_{Z\times Z}, |\bd \xi|_{Z\times Z}\cup J_{Z\times Z})&\rTo^{\bd_*}&H_{i-1}(|\bd \xi|_{Z\times Z}\cup J_{Z\times Z}, J_{Z\times Z})\\
\uTo^{\D}&&\uTo^{\D}\\
 H^{2n-i}(Z\times Z-(|\bd \xi|_{Z\times Z}\cup J_{Z\times Z}),Z\times Z-(|\xi|_{Z\times Z}\cup J_{Z\times Z}))&\rTo^{d^*}& H^{2n-i+1}(Z\times Z-( J_{Z\times Z}),Z\times Z-(|\bd \xi|_{Z\times Z}\cup J_{Z\times Z}))\\
\dTo^{\Delta^*} &&\dTo^{\Delta^*} \\
 H^{2n-i}(Z-|\bd \xi|_Z'\cup J_Z,Z-(|\xi|_Z'\cup J_Z))&\rTo^{d^*}& H^{2n-i+1}(Z-J_Z,Z-(|\bd \xi|_Z'\cup J_Z))\\
\dTo^{\D}&&\dTo^{\D}\\
H_{i-n}(|\xi|_Z'\cup J_Z,|\bd \xi|_Z'\cup J_Z)&\rTo^{\bd_*}&H_{i-n-1}(|\bd \xi|_Z'\cup J_Z,J_Z)\\
\dTo^{\bd_*}&&\dTo^=\\
H_{i-n-1}(|\bd \xi|_Z'\cup J_Z, J_Z)&\rTo^=&H_{i-n-1}(|\bd \xi|_Z'\cup J_Z,J_Z)\\
&\rdTo&\dTo\\
&&H_{i-n-1}(|\bd \xi|'_Z\cup J_Z,(|\bd\xi|_Z'\cup J_Z)-int(\tau))\\
 &&\uTo^{\cong}\\
&&H_{i-n-1}(\tau, \bd \tau)\\
&&\uTo^\cong\\
&&\Z.
\end{diagram}}

\noindent Except for the squares involving $\D$, this diagram commutes by the standard properties of connecting morphisms. The top square involving $\D$ commutes up to $(-1)^{2n}=1$ by  Lemma  \ref{L: GM dual boundary}, while the lower square involving $\D$ commutes up to $(-1)^n$ by the same lemma. Starting up the upper left and proceeding right then down in the diagram takes the homology class representing $\xi$ to $\Delta^!(\bd \xi)$ and then to its coefficient for $\tau$. On the other hand, proceeding straight downward takes us to the class representing $\Delta^{Z,T}_!(\xi)$ and then to its boundary. More specifically, applying Lemma \ref{L: useful coeff big} with $A=|\xi|'_Z\cup J_Z$, $B=|\bd \xi|'_Z\cup J_Z$, and $C=J_Z$, the $\bd_*$ map here can be thought of as taking   $\Delta^{Z,T}_!(\xi)$ to the piece of its boundary consisting of simplices in $|\bd \xi|'_Z$ but not in $J_Z$. We know that $\tau$ is such a simplex by assumption. The bottom of the diagram then provides the coefficient of $\tau$ according to Remark \ref{R: useful coeff big}. 

It follows that $\tau$ obtains the same coefficient, up to the sign $(-1)^n$, by either route, proving the proposition. 
\end{proof}

The next lemma concerns the behavior of $\Delta_!$ under restriction. Recall\footnotemark that if $X$ is a PL space then any open subset $U$ is also PL and that if $X$ is given a triangulation $T$ then $U$ can be given a triangulation $S$ such that every simplex of $S$ is contained in a simplex of $T$ (see \cite[Section I.1.3]{Bo}). This induces a restriction map $c^{T,\infty}_*(X)\to c_*^{S,\infty}(U)$ that takes each $i$-simplex of $T$ to the formal sum of the compatibly oriented $i$-simplices of $S$ contained in it. Such maps then induce a restriction map $r: C^\infty_*(X)\to C^\infty_*(U)$. 

\footnotetext{\label{F: sub} Here's one way to do this: start with a triangulation $T$ of $X$ and an arbitrary triangulation $S_0$ of $U$ compatible with the PL structure on $U$ inherited from $X$. We can then obtain a cell complex on $U$ whose cells are the intersections of the simplices of $S_0$ with the simplices of $T$ (see \cite[Example 2.8.5]{RS}). Finally, we can subdivide the cell complex on $U$ to a simplicial complex $S$ triangulating $U$ by using the procedure of  \cite[Proposition 2.9]{RS}; note that the cited proposition works on infinite complexes either by choosing a well-ordering of the vertices or simply by using a partial ordering on the vertices which is locally a total ordering on each cell. Each simplex of $S$ will be contained in a cell of the cell complex and thus in a simplex of $T$. }

\begin{lemma}\label{L: restrict umkehr}
Let $U\subset X$ be an open subset. Then there is a commutative diagram
\begin{diagram}
C^{\Delta,\infty}_*((X,\Sigma_X)\times (X,\Sigma_X))&\rTo^{\Delta_!}& C_{*-n}(X,\Sigma_X)\\
\dTo^r&&\dTo^r\\
C^{\Delta,\infty}_*((U,\Sigma_U)\times (U,\Sigma_U))&\rTo^{\Delta_!}& C_{*-n}(U,\Sigma_U).
\end{diagram}
\end{lemma}
\begin{proof}
Let $\xi\in C^{\Delta,\infty}_i((X,\Sigma_X)\times (X,\Sigma_X))$, let  $T$ be a triangulation $T$ of $X$  satisfying our usual conditions as given above, and let $S$ be a triangulation of $U$ such that  every simplex of $S$ is contained in some simplex of $T$. Let $\sigma$ be any $i-n$ simplex of $S$ contained in $|\xi|'$ but not contained in $\Sigma_U$. As $\dim(|\xi|'\leq i-n)$ and as $|\xi|'$ is triangulated by $T$, any such $\sigma$ is contained in an $i-n$ simplex $\gamma$ of $|\xi|'$ in $T$.
By the arguments of Remark \ref{R: subdiv}, we see that the coefficient of $\sigma$ in $r\Delta_!(\xi)$ will equal $I_\gamma(\xi)$ (which can be computed as $I^{Z,T}_\gamma(\xi)$ using $T$ and some appropriate choice of $Z$). We must show that this is equal to the coefficient of $\sigma$ in  $\Delta_!r(\xi)$, which is $I_\sigma(r(\xi))$ (which can be computed as $I^{Y,S}_\gamma(\xi)$ using $S$ and some appropriate choice of $Y$). 
As  $\sigma$ in $S$ is an arbitrary simplex of $|\xi|'\cap U$, this will suffice to prove the lemma. 

Next, we take a sufficiently iterated barycentric subdivision $T'$ of $T$ so that
\begin{enumerate}
\item there is an $i-n$ simplex $\tau$ of $T'$ contained within $\sigma$ and
\item the star neighborhood of $\tau$ in $T'$ is contained in $U$. 
\end{enumerate}
This is possible by the arguments of \cite[Section 15]{MK}. Then let $S'$ be a subdivision of $S$ such that $\tau$ is also a simplex of $S'$ and such that every simplex of $S'$ is contained in some simplex of $T'$. We can find such an $S'$ by the same procedure described in Footnote \ref{F: sub}: form the cell complex on $U$ whose cells are the intersections of the simplices of $S$ and $T'$ (see \cite[Example 2.8.5]{RS}) and then create the simplicial complex obtained from this cell complex but with the same vertices via the procedure of \cite[Proposition 2.9]{RS}. As $\tau$ is already a simplex of $T'$ contained in the simplex $\sigma$ of $S$, $\tau$ will be a cell of the cell complex and then remain a simplex in the triangulation $S'$. As every simplex of $S'$ is contained in a simplex of $T'$, the star neighborhood of $\tau$ in $S'$ will be a subset of the star neighborhood of $\tau$ in $T'$.

Now, let the star neighborhood of $\tau$ in $T'$ be our space $Z$. As $Z\subset U$, it will also be triangulated as a union of $n$-simplexes in $S'$. We will use this same $Z$ for the computation of the intersection coefficients $I^{Z,T'}_{\tau}(\xi)$ in $X$ and $I^{Z,S'}_{\tau}(r(\xi))$ in $U$. By Remark \ref{R: subdiv}, 
$I^{Z,T'}_{\tau}(\xi)=I_\gamma(\xi)$, as $\tau\subset \sigma\subset\gamma$,  and $I^{Z,S'}_{\tau}(r(\xi))=I_{\sigma}(r(\xi))$. So it suffices to show that $I^{Z,T'}_{\tau}(\xi)=I^{Z,S'}_{\tau}(r(\xi))$.
This follows from the following commutative diagram and the definitions of these intersection coefficients, noting that the spaces $Z$, $Z\times Z$, $J_Z$, $J_{Z\times Z}$, and $D_{Z\times Z}$ (as a subspace of $X\times X$) do not depend on the particular choice of triangulation used to define $Z$:

{\footnotesize
\begin{diagram}
H_i^{\infty}(|\xi|,|\bd \xi|)&\rTo^r&H_i^{\infty}(|r(\xi)|,|r(\bd \xi)|))\\
\dTo&&\dTo\\
H_i^{\infty}(|\xi|\cup J_{Z\times Z}\cup D_{Z\times Z},|\bd \xi|\cup J_{Z\times Z}\cup D_{Z\times Z})&\rTo^r&H_i^{\infty}(U\cap(|\xi|\cup J_{Z\times Z}\cup D_{Z\times Z}),U\cap(|\bd \xi|\cup J_{Z\times Z}\cup D_{Z\times Z}))\\
\uTo^\cong&&\uTo^\cong\\
H_i(|\xi|_{Z\times Z}\cup J_{Z\times Z}, |\bd \xi|_{Z\times Z}\cup J_{Z\times Z})&\rTo^=&H_i(|\xi|_{Z\times Z}\cup J_{Z\times Z}, |\bd \xi|_{Z\times Z}\cup J_{Z\times Z}).
\end{diagram}}

Note that the excisions are both still valid as all of the spaces in the bottom row are contained within $U$. 
\end{proof}

\subsection{The cross product}\label{S: cross}

In addition to the umkehr map, the other main ingredient of the intersection product is the PL chain cross product $\epsilon: C_*^\infty(X)\otimes C_*^\infty(X)\to C_*^{\infty} (X\times X)$; see \cite{McC}, \cite{GBF18}, or \cite[Section 5.2.2]{GBF35}. The cross products in the citations are defined primarily for compact chains, but as the cross product construction can be defined locally using simplicial chain representatives, there is no difficulty extending it in the obvious way to locally finite chains. This map induces a relative product $\epsilon: C_*^\infty(X,\Sigma)\otimes C_*^\infty(X,\Sigma)\to C_*^\infty((X,\Sigma)\times (X,\Sigma))$.  

When an element of  $C_*^\infty(X)\otimes C_*^\infty(X)$ can be represented as $\xi\otimes \eta$, we sometimes use the notation $\epsilon(\xi\otimes \eta)=\xi\times \eta$. 

\begin{lemma}\label{L: restrict cross}
Let $U\subset X$ be an open subset. Then there is a commutative diagram
\begin{diagram}
C_*^\infty(X)\otimes C_*^\infty(X)&\rTo^{r\otimes r}& C_*^\infty(U)\otimes C_*^\infty(U)\\
\dTo^\epsilon&&\dTo^\epsilon\\
C^{\infty}_*(X\times X)&\rTo^r&C_*^\infty(U\times U),
\end{diagram}
where the horizontal maps are induced by restriction.
This induces a diagram 
\begin{diagram}
C_*^\infty(X,\Sigma_X)\otimes C_*^\infty(X,\Sigma_X)&\rTo^{r\otimes r}& C_*^\infty(U,\Sigma_U)\otimes C_*^\infty(U,\Sigma_U)\\
\dTo^\epsilon&&\dTo^\epsilon\\
C^{\infty}_*((X,\Sigma_X)\times (X,\Sigma_X))&\rTo^r&C_*^\infty((U,\Sigma_U)\times (U,\Sigma_U)),
\end{diagram}
\end{lemma}
\begin{proof}
The groups $C_*^\infty(X)\otimes C_*^\infty(X)$ in any fixed degree are generated by elements of the for $\xi\otimes \eta$ (though not every element can be written in this form). Suppose $\xi$ is an $i$-chain and $\eta$ is a $j$-chain. 
As observed in \cite[Section II.1]{Bo}, a PL chain $\xi$ is completely determined by $|\xi|$, $|\bd \xi|$, and the local sheaf values $[\xi]_x\in (\mc H_i)_x\cong H_i(|\xi|,|\xi|-\{x\})\cong  H^\infty_i(|\xi|,|\xi|-\{x\})$ determined by the class $[\xi]\in H_i(|\xi|,|\bd \xi|)$ as $x$ varies in $|\xi|-|\bd \xi|$. 

Noting that $|r(\xi)|=|\xi|\cap  U$ and as $|\xi\times \eta|=|\xi|\times |\eta|$, we have 
\begin{align*}
|r(\xi\times \eta)|&=|\xi\times \eta|\cap (U\times U)\\
&=(|\xi|\times |\eta|)\cap (U\times U)\\
&=(|\xi|\cap U)\times (|\eta|\cap U)\\
&=|r(\xi)|\times |r(\eta)|,
\end{align*}
and so $r\epsilon(\xi\otimes \eta)$ and $\epsilon(r(\xi)\otimes r(\eta))$ have the same support. 

Similarly, as $r$ and $\epsilon$ are chain maps and $|\bd(\xi\otimes \eta)|=(|\bd \xi|\times |\eta|)\cup (|\xi|\times |\bd \eta|)$, we have
\begin{align*}
|\bd r\epsilon(\xi\otimes \eta)|&=|r\epsilon((\bd \xi)\otimes \eta\pm  \xi\otimes \bd \eta)|\\
&=[(|\bd \xi|\times |\eta|)\cup (|\xi|\times |\bd \eta|)]\cap(U\times U)\\
&=(|r(\bd \xi)|\times |r(\eta)|)\cup (|r(\xi)|\times |r(\bd \eta)|)\\
&=(|\bd r( \xi)|\times |r(\eta)|)\cup (|r(\xi)|\times |\bd r(\eta)|)\\
&=|\bd( r(\xi)\times r(\eta))|\\
&=|\bd \epsilon(r\otimes r)(\xi\times \eta)|.
\end{align*}
Do the boundaries of $r\epsilon(\xi\otimes \eta)$ and $\epsilon(r(\xi)\otimes r(\eta))$ have the same support.

It remains to check the local sheaf values, but clearly, via excisions and using standard local chain representatives, we have a commutative diagram at each point $(x,y)\in |r(\xi\times \eta)|-|\bd r(\xi\times \eta)|$:

\begin{diagram}
H_i(|\xi|, |\xi|-\{x\})\otimes H_j(|\eta|, |\eta|-\{y\})&\rTo^\cong &H_i(|r\xi|, |r\xi|-\{x\})\otimes H_j(|r\eta|, |r\eta|-\{y\})\\
\dTo^\epsilon&&\dTo^\epsilon\\
H_{i+j}(|\xi|\times |\eta|,|\xi|\times |\eta|-\{(x,y)\})&\rTo^\cong&H_{i+j}(|r\xi|\times |r\eta|,|r\xi|\times |r\eta|-\{(x,y)\})
\end{diagram}

\end{proof}

\subsection{The chain-level intersection product}\label{S: pairing}

We can now define our intersection product:

\begin{definition}\label{D: int pairing}
Let $X$ be an oriented $n$-dimensional PL stratified pseudomanifold.
Let the \emph{domain} $\mf G^\infty_*(X,\Sigma)$ be the subcomplex of $C_*^\infty(X,\Sigma)\otimes C_*^\infty(X,\Sigma)$ defined by $$\mf G^\infty_*(X,\Sigma)=\epsilon^{-1}(C_*^{\Delta,\infty}((X,\Sigma)\times (X,\Sigma))).$$ 
 The \emph{intersection product} is defined to be
$$\mu=\Delta_!\circ \epsilon:\mf G^\infty_*(X,\Sigma)\to C_{*-n}^\infty(X,\Sigma).$$ 
\end{definition}

We note that $\mu$ is a chain map of degree $-n$ as  $\Delta_!$ is a chain map of degree $-n$ by Propositions \ref{P: chain map} and \ref{P: additivity}, while $\epsilon$ is a degree $0$ chain map by  \cite[Corollary 5.2.10]{GBF35}.

\begin{remark}\label{R: McC comp}
We will see below in Proposition \ref{P: compact} that when $X$ is compact we can let every $Z$ in the definition of $\Delta_!$ be $X$ itself and then $\mu$ reduces to the map $\mu_2$ of \cite{GBF18} up to the convention changes noted in the introduction. If $X$ is moreover a compact oriented PL manifold, i.e.\ $\Sigma=\emptyset$, then this map is isomorphic to the map $\mu_2$ of McClure \cite{McC}, again up to conventions. This isomorphism is not quite immediate due to some slight differences between the definitions of the umkehr maps in \cite{McC} and \cite{GBF18} (in particular, the uses of $\mf D$ versus $\D$), but it is not hard to show that they're isomorphic using the tools of Section \ref{S: dualities} and the discussion of McClure's construction in \cite[Section 5]{McC}. By \cite[Section 4.3]{GBF18}, these maps $\mu_2$ induce the Goresky-MacPherson intersection product.
\end{remark}

\begin{proposition}\label{P: qi}
The inclusion $\mf G^{\infty}_*(X,\Sigma)\to C_*^\infty(X,\Sigma)\otimes C_*^\infty(X,\Sigma)$ induces an isomorphism on homology.
\end{proposition}
\begin{proof}
This follows directly as a special case of the proof Theorem 3.5 of \cite{GBF18}, which is the equivalent statement for compact chains. As noted in \cite[Remark 3.8]{GBF18}, that proof consists of pushing chains into stratified general position 
(see Definition \ref{D: sgp}, below)  by proper isotopies, and so the arguments work just as well for locally finite chains. The proof in \cite{GBF18} also deals only with absolute chains, not relative chains mod $\Sigma$, but this only means that the proof there is stronger than necessary here, providing stronger general position in the singular strata than is strictly needed to satisfy Definition \ref{D: Delta}.
\end{proof}

This intersection product is also compatible with restriction: 

\begin{proposition}\label{P: res}
If $U\subset X$ is an open subset then we obtain a commutative diagram

\begin{diagram}
 C_*^\infty(X,\Sigma_X)\otimes C_*^\infty(X,\Sigma_X)&\lInto&\mf G^\infty_*(X,\Sigma_X)&\rTo^{\mu}& C_{*-n}^\infty(X,\Sigma_X)\\
\dTo&&\dTo&&\dTo\\
 C_*^\infty(U,\Sigma_U)\otimes C_*^\infty(U,\Sigma_U)&\lInto&\mf G^\infty_*(U, \Sigma_U)&\rTo^{\mu}& C_{*-n}^\infty(U,\Sigma_U),
\end{diagram}
where the vertical maps are induced by restriction.
\end{proposition}
\begin{proof}
Commutativity on the right follows immediately by putting together Lemmas \ref{L: restrict umkehr} and \ref{L: restrict cross}. Commutativity on the left follows from Lemma \ref{L: restrict cross} and the fact that the restriction map takes $C^{\Delta,\infty}_*((X,\Sigma_X)\times (X,\Sigma_X))$ to $C^{\Delta,\infty}_*((U,\Sigma_U)\times (U,\Sigma_U))$.
\end{proof}

Together, Definition \ref{D: int pairing}, Propositions \ref{P: qi} and \ref{P: res}, and Remark \ref{R: McC comp} demonstrate all the claims of Theorem \ref{T: main} concerning the intersection product \eqref{E1}.

\section{The intersection product for intersection homology}\label{S: int pairings}

In this section, we briefly review the definition of intersection chains and then demonstrate how an intersection product of intersection chains follows from the work of the previous sections. A thorough introduction to intersection homology can be found in \cite{GBF35}. 

\subsection{PL intersection chains}\label{S: IC}

If $X$ is a PL stratified pseudomanifold, a \emph{perversity} on $X$ is a function $$\bar p:\{\text{singular strata of $X$}\}\to \Z.$$ Recall from Section \ref{S: PL chains} that if $\xi$ is a chain then $|\xi|$ denotes the union in some triangulation (compatible with $\xi$ and with the stratification of $X$) of the $i$-simplices of $\xi$ that are not contained in $\Sigma$.
The \emph{perversity $\bar p$ PL intersection chain complex} $I^{\bar p}C_*(X)$ is defined as the subcomplex of  $C_*(X,\Sigma)$ consisting of chains  $\xi\in C_i(X,\Sigma)$ such that the following conditions are satisfied for each singular stratum $\mc Z$ of $X$:
\begin{enumerate}
\item $\dim(|\xi|\cap \mc Z)\leq i-\codim(\mc Z)+\bar p(\mc Z)$,
\item $\dim(|\bd \xi|\cap \mc Z)\leq i-1-\codim(\mc Z)+\bar p(\mc Z)$.
\end{enumerate}
The groups $I^{\bar p}C_*(X)$  form a chain complex due to the second condition, and the homology groups are the \emph{intersection homology groups}, denoted $I^{\bar p}H_*(X)$.  Similarly, $I^{\bar p}C^\infty_*(X)$ and  $I^{\bar p}H^\infty_*(X)$ are defined identically by imposing the dimension conditions on chains in $C_*^\infty(X,\Sigma)$. 

The reader familiar with the original definition of intersection chains, e.g.\ from Goresky-MacPherson \cite{GM1}, might be surprised to see them defined here as a subcomplex of the  \emph{relative} chain complex, but it is shown in \cite[Section 6.2]{GBF35} that these intersection chain complexes  are isomorphic to those of Goresky and MacPherson provided that $\bar p(\mc Z)\leq \codim(\mc Z)-2$ for all singular strata $\mc Z$, which is the case for all perversities treated in \cite{GM1}. Roughly speaking, for these stricter perversities the chains are sufficiently forced away from  $\Sigma$ that, even though they may still intersect $\Sigma$, they are not affected by what happens internal to $\Sigma$. Conversely, for larger perversities, the relative approach is critical for obtaining agreement with sheaf-theoretic intersection homology; see  \cite{GBF26, GBF35} for expository discussions.

\subsection{The intersection product}\label{S: IC pairing}

In this section, we demonstrate how the intersection product $\mu$ of the prior section restricts to provide an intersection product on intersection chains. In addition to limiting ourselves to intersection chains, we will need to impose a stronger notion of general position (cf.\ Remark \ref{R: weak GP}). This \emph{stratified general position} is formulated in general in \cite[Definition 3.1]{GBF18}; we here restrict to the version we will need:

\begin{definition}\label{D: sgp}
Let $X$ be a PL stratified pseudomanifold and $\Delta:X\to X\times X$ the diagonal inclusion $\Delta(x)=(x,x)$. We will say that the PL subset $A\subset X\times X$ is in \emph{stratified general position} (with respect to $\Delta$) if for each stratum $\mc Z\subset X$ we have 
$$\dim(A\cap \Delta(\mc Z))\leq \dim(A\cap (\mc Z\times \mc Z))-\dim(\mc Z).$$
In other words, noting that $\dim(\mc Z\times \mc Z)=2\dim(\mc Z)$, we require that, for each stratum $\mc Z$, the PL spaces $A\cap (\mc Z\times \mc Z)$ and $\Delta(\mc Z)$ are in general position in the manifold $\mc Z\times \mc Z$.

If $\xi\in C^{\infty}_i((X,\Sigma)\times (X,\Sigma))$, then we will say that $\xi$ is in stratified general position if $|\xi|$ is in stratified general position.  
\end{definition}

\begin{remark}\label{R: GP int}
Notice that if $\xi\in C^{\infty}_i((X,\Sigma)\times (X,\Sigma))$ is in stratified general position and if $R$ is the union of  regular strata of $X$, then we have $\dim(\Delta^{-1}(|\xi|)-\Sigma)=\dim(|\xi|\cap \Delta(R))\leq i-n$. So if $|\xi|$ and $|\bd\xi|$ are in stratified general position, then $\xi\in C_i^{\Delta,\infty}((X,\Sigma)\times (X,\Sigma))$. 
\end{remark}

Next, let $P=(\bar p_1,\bar p_2)$ be a pair of perversities.

\begin{definition}\label{D: GP}
The \emph{domain}  $G^{\infty,P}_*(X)$ of the intersection product  is defined to be the subcomplex of $I^{\bar p_1}C_*^\infty(X)\otimes I^{\bar p_1}C_*^\infty(X)$ consisting of chains $\xi$ such that $|\epsilon(\xi)|$ and $|\epsilon(\bd \xi)|$ are in stratified general position. We have $G^{\infty,P}_*(X)\subset \mf G^{\infty}_*(X, \Sigma)$ by Remark \ref{R: GP int}.
Let $\mu:G^{\infty,P}_*(X)\to C_{*-n}^\infty(X,\Sigma)$ be the restriction of the intersection product. 
\end{definition}

Now that we have constructed the intersection product for intersection chains, we can verify its key properties:

\begin{proposition}\label{P: iqi}
The inclusion $G^{\infty,P}_*(X)\to I^{\bar p_1}C_*^\infty(X)\otimes I^{\bar p_2}C_*^\infty(X)$ induces an isomorphism on homology.
\end{proposition}
\begin{proof}
As for Proposition \ref{P: qi}, this result follows from the same reasoning as the analogous Theorem 3.7 of \cite{GBF18}.
\end{proof}

\begin{proposition}\label{P: G}
Let $\bar p_1+\bar p_2$ be the perversity that  evaluates on the singular stratum $\mc Z$ to $\bar p_1(\mc Z)+\bar p_2(\mc Z)$.
The intersection product $\mu$ restricts to a map $\mu: G^{\infty,P}_*(X)\to I^{\bar p_1+\bar p_2}C_{*-n}^\infty(X)$. 
\end{proposition}
\begin{proof}
We already know that $\mu:G^{\infty,P}_*(X) \to C_{*-n}^{\infty}(X, \Sigma)$ is well defined, so we have only to show that the image lies in $I^{\bar p_1+\bar p_2}C_{*-n}^\infty(X)$. 

As $G^{\infty,P}_i(X)\subset I^{\bar p_1}C_*^\infty(X)\otimes I^{\bar p_2}C_*^\infty(X)$, each element $\xi\in G^{\infty,P}_i(X)$ can be written as a sum $\sum_a \zeta_a\otimes \eta_a$ with $\zeta_a\in I^{\bar p_1}C_*^\infty(X)$ and $\eta_a\in I^{\bar p_1}C_*^\infty(X)$. Suppose $\zeta_a\in I^{\bar p_1}C_j^\infty(X)$ and $\eta_a\in  I^{\bar p_2}C_k^\infty(X)$ with $j+k=i$, and let $\mc Z$ be a stratum of $X$. Then, by definition, $\dim(|\zeta_a|\cap \mc Z)\leq j-\codim_X(\mc Z)+\bar p_1(\mc Z)$ and similarly for $\eta_a$. Thus 
\begin{align*}
\dim(|\epsilon(\zeta_a\otimes \eta_a)|\cap (\mc Z\times \mc Z))&= \dim((|\zeta_a|\times |\eta_a|)\cap (\mc Z\times \mc Z))\\
&= \dim((|\zeta_a|\cap \mc Z)\times (|\eta_a|\cap \mc Z))\\
&= \dim(|\zeta_a|\cap \mc Z)+\dim(|\eta_a|\cap \mc Z)\\
&\leq j-\codim_X(\mc Z)+\bar p_1(\mc Z)+k-\codim_X(\mc Z)+\bar p_2(\mc Z)\\
&=i-2\codim_X(\mc Z)+\bar p_1(\mc Z)+\bar p_2(\mc Z).
\end{align*}
It follows that $\dim(|\epsilon(\xi)|\cap (\mc Z\times \mc Z))\leq i-2\codim_X(\mc Z)+\bar p_1(\mc Z)+\bar p_2(\mc Z)$. Therefore, as $|\epsilon(\xi)|$ is in stratified general position by assumption, 
\begin{align*}
\dim(|\epsilon(\xi)|\cap \Delta(\mc Z))&\leq i-2\codim_X(\mc Z)+\bar p_1(\mc Z)+\bar p_2(\mc Z)-\dim(\mc Z)\\
&=i-2\codim_X(\mc Z)+\bar p_1(\mc Z)+\bar p_2(\mc Z)-(n-\codim_X(\mc Z))\\
&=i-n-\codim_X(\mc Z)+\bar p_1(\mc Z)+\bar p_2(\mc Z).
\end{align*}
Therefore, as $\mu(\xi)$ is supported in $\Delta^{-1}(|\epsilon(\xi)|)=|\epsilon(\xi)|\cap \Delta(X)$, we see that $\mu(\xi)$ satisfies the conditions to be $\bar p_1+\bar p_2$ allowable. An equivalent computation shows that  $\bd\mu(\xi)=\mu(\bd\xi)$ is also $\bar p_1+\bar p_2$ allowable, so $\mu(\xi)\in I^{\bar p_1+\bar p_2}C_{*-n}^\infty(X)$.
\end{proof}

\begin{remark}
Notice that the full power in Proposition \ref{P: G} of the stratified general position assumption is in making sure that $\mu(\xi)$ lies in the intersection chain complex with the smallest possible perversity, $\bar p_1+\bar p_2$. Consequently, $\mu(\xi)$ also lies in $I^{\bar r}C_{*-n}^\infty(X)$ for all perversities $\bar r$ with $\bar p_1+\bar p_2\leq r$. In fact, it follows directly from the definitions that $I^{\bar p}C_{*}^\infty(X)\subset I^{\bar q}C_{*}^\infty(X)$ if $\bar p\leq \bar q$, meaning that $\bar p(\mc Z)\leq \bar q(\mc Z)$ for all singular strata $\mc Z$. If our goal is simply to have $\mu(\xi)\in  I^{\bar r}C_{*-n}^\infty(X)$ for some $\bar r\geq \bar p_1+\bar p_2$, we could relax the stratified general position requirement accordingly on the singular strata. In fact, this is the case in the original definition of the intersection product for intersection chains in \cite[Section 2.1]{GM1}, where, in addition to general position on the regular strata, it is only required for the intersection of a pair of chains to satisfy the dimension requirements recalled in  Section \ref{S: IC} with respect to $\bar r$ (see \cite{GM1} for full details). 
\end{remark}

\begin{proposition}\label{P: ires}
If $U\subset X$ is an open subset then there is a  commutative diagram
\begin{diagram}
I^{\bar p_1}C_*^\infty(X)\otimes I^{\bar p_2}C_*^\infty(X)&\lInto& G^{\infty,P}_*(X)&\rTo^{\mu}& I^{\bar p_1+\bar p_2}C_{*-n}^\infty(X)\\
\dTo^r&&\dTo^r&&\dTo^r\\
 I^{\bar p_1}C_*^\infty(U)\otimes I^{\bar p_2}C_*^\infty(U)&\lInto&G^{\infty,P}_*(U)&\rTo^{\mu}& I^{\bar p_1+\bar p_2}C_{*-n}^\infty(U),
\end{diagram}
where the vertical maps are induced by restriction.

\end{proposition}
\begin{proof}
This follows from Lemmas \ref{L: restrict umkehr} and \ref{L: restrict cross} as in the proof of Proposition \ref{P: res}, together with the observation that the restriction maps
$r:C_{*}^\infty(X,\Sigma_X) \to C^\infty_{*}(U,\Sigma_U)$ take the subcomplex  $I^{\bar p}C_*^\infty(X)$ to $I^{\bar p}C_*^\infty(U)$ for any perversity $\bar p$. 
\end{proof}

Next, let us verify that our intersection chain product agrees with the Goresky-MacPherson product on compact spaces, which we denote by $\pf$. 
For this, we need the notion of stratified general position for two chains.

\begin{definition}\label{D: genpos}
We say that two chains $\xi,\eta$ in $X$ are in stratified general position if $\epsilon(\xi\otimes\eta)$ is in stratified general position, i.e.\ if $|\xi|\times |\eta|$ is in stratified general position with respect to $\Delta$. As $(|\xi|\times |\eta|)\cap \Delta(\mc Z)=|\xi|\cap |\eta|\cap \mc Z$ and $\dim(B\times C)=\dim(B)+\dim(C)$ in general (if neither $B$ nor $C$ is empty), this is equivalent to requiring that $|\xi|\cap \mc Z$ and $|\eta|\cap \mc Z$ are in general position in $\mc Z$, i.e.\ $\dim(|\xi|\cap |\eta|\cap \mc Z)\leq \dim(|\xi|\cap \mc Z)+\dim(|\eta|\cap \mc Z)-\dim(\mc Z)$.
\end{definition}

\begin{proposition}\label{P: compact}
Suppose $X$ is a compact oriented PL stratified pseudomanifold and that $\bar p_1,\bar p_2$ are Goresky-MacPherson perversities such that there exists a Goresky-MacPherson perversity $\bar r$ with $\bar p_1+\bar p_2\leq \bar r$.  
Then the composition $\mu:G^P_*(X)\xr{\mu}I^{\bar p_1+\bar p_2}C_{*-n}(X)\to I^{\bar r}C_{*-n}(X)$ agrees with the intersection product $\mu_2$ defined in \cite{GBF18}, replacing the Goresky-MacPherson duality map with the one constructed here. Therefore, if $\xi\in I^{\bar p_1}C_i(X)$ and $\eta\in I^{\bar p_2}C_j(X)$ are such that the pairs\footnote{We do not need to require that the pair $(\bd \xi,\bd\eta)$ be in stratified position, as we do not here run into the logical difficulty mentioned in Remark \ref{R: GM bdbd}. This is because $(\xi,\bd \eta)$ and $(\bd \xi,\eta)$ being in stratified general position is enough to guarantee that 
$\bd (\xi\times \eta)=(\bd \xi)\times \eta\pm \xi\times \bd \eta$ is in stratified general position with respect to $\Delta$, and we never need to consider directly terms such as $(\bd \xi)\pf \eta$ or $\xi\pf\bd \eta$.} 
$(\xi,\eta)$, $(\xi,\bd \eta)$, and $(\bd \xi,\eta)$ are in stratified general position, then $\mu(\xi\otimes \eta)$ is equal to the Goresky-MacPherson intersection product  $\xi\pf \eta$, up to sign.
\end{proposition}

\begin{remark}
Goresky-MacPherson perversitities are the perversities that satisfy the original definition of a perversity in \cite{GM1}; see also \cite[Definition 3.1.4]{GBF35}.
The restriction on perversities in the statement of the proposition is necessary because the intersection products in \cite{GM1,GBF18} are defined only for Goresky-MacPherson perversities. However, the definitions of those products work just as well so long as all perversities satisfy $\bar p(\mc Z)\leq \codim(\mc Z)-2$ for all singular strata $\mc Z$. To use perversities that allow for $\bar p(\mc Z)>\codim(\mc Z)-2$ for some strata it is necessary to utilize the relative chain complexes as we have done here. For an explanation of the different behaviors resulting from these assumptions, see \cite[Chapter 6]{GBF35}.
\end{remark}

\begin{proof}[Proof of Proposition \ref{P: compact}]
Given that $X$ is compact, the entire construction of $\Delta_!$ in Sections \ref{S: local umkehr} and \ref{S: umkehr} can be carried out with a single neighborhood space $Z$, namely $X$ itself. In this case, $D_Z=D_{Z\times Z}=S_Z=\emptyset$, so $J_Z=\Sigma_Z$ and $J_{Z\times Z}=(X\times \Sigma_X)\cup(\Sigma_X\times X)=\Sigma_{X\times X}$. So then if $\xi\in C_*^{\Delta}((X,\Sigma)\times (X,\Sigma))$, by following Definition \ref{D: int coeff Z} in this case, the chain $\Delta_!(\xi)$  corresponds under the isomorphism of Lemma \ref{L: useful coeff big} to  the image of $[\xi]$ under the composition 
\begin{align*}
H_i(|\xi|,|\bd \xi|)
&\to H_i(|\xi|\cup \Sigma_{X\times X},|\bd \xi|\cup \Sigma_{X\times X})\\
&\xleftarrow{\D} H^{2n-i}(X\times X-(|\bd \xi|\cup \Sigma_{X\times X}),X\times X-(|\xi|\cup \Sigma_{X\times X}))\\
&\xr{\Delta^*}  H^{2n-i}(X-|\bd \xi|'\cup \Sigma_{X},X-(|\xi|'\cup  \Sigma_{X}))\\
&\xr{\D} H_{i-n}(|\xi|'\cup \Sigma_X,|\bd \xi|'\cup \Sigma_X).
\end{align*}
But, up to indexing shifts and the new corrected version of the map $\D$, this is precisely the map defined in Section 4.2 of \cite{GBF18}. The cross product part of the intersection product remains identical.

The comparison with the Goresky-MacPherson intersection product is then outlined\footnote{That argument notes $(\alpha\times \beta)\frown (\xi\times \eta)=(-1)^{|\beta||\xi|}(\alpha\times \beta)\frown (\xi\times \eta)$, as provided by \cite[\S VII.12.17]{Dold}, but does not trace through the full compatibility between cross products and the Dold and Goresky-MacPherson duality maps. However, that is a routine, though tedious, exercise in applying naturality of cross products, as well as the formula just quoted, to the defining diagrams of the duality isomorphisms. }
 in the proof of Proposition 4.9 of \cite{GBF18}.  
\end{proof}

Definition \ref{D: GP} and Propositions \ref{P: iqi}, \ref{P: G}, \ref{P: ires}, and \ref{P: compact} complete the proof of Theorem \ref{T: main}.\qedhere
\qedsymbol

\section{Duality of intersection products and cup products}\label{S: cup dual}

In this section, we utilize some of the tools we have developed to verify a seemingly well-known fact: On a compact oriented PL manifold, the Goresky-MacPherson homology intersection product and the cup product are Poincar\'e dual. We also show that the Goresky-MacPherson homology product agrees with Dold's \cite[Section VIII.13]{Dold}.
 We then discuss how our argument fails to generalize to compact oriented PL stratified pseudomanifolds, demonstrating the need for the ``sheafifiable'' intersection product we have developed in this paper. 

We begin with the following result about PL manifolds; it is a nice corollary of Lemma \ref{L: natural dual}, which demonstrated the naturality of the duality map $\D$. We will use the following notation: For a subset $A\subset M$, let $\mf i_A: H_i( A)\to H_i(M)$ and $\mf i^A: H^i(M, M-A) \to H^i(M)$ be induced by inclusion.

\begin{corollary}\label{C: duals}
Let $M$ be a compact oriented PL $n$-manifold, and let $\xi$ be a PL $i$-cycle in $M$.  The following diagram commutes:

\begin{diagram}
H^{n-i}(M,M-|\xi|)&\rTo^{\mf i^{|\xi|}} &H^{n-i}(M)\\
\dTo^\D&&\dTo^\D\\
H_i(|\xi|)&\rTo^{\mf i_{|\xi|}}&H_i(M).
\end{diagram}
In other words, $\mf i_{|\xi|}\D=\D\mf i^{|\xi|}$.
\end{corollary}
\begin{proof}
In Lemma \ref{L: natural dual}, take $X=K'=M$, $K=|\xi|$, and $L=L'=\emptyset$. 
\end{proof}

Tracing through the definitions, the vertical map on the right of the diagram in Corollary \ref{C: duals} is just the (signed) cap product with the fundamental class. The corollary thus demonstrates a compatibility between classical Poincar\'e duality and the dualities we've been using to explore chains: If $\xi$ is an $i$-cycle in $M$, it can be represented by a class  $[\xi]\in H_i(|\xi|)$ by Lemma \ref{L: useful}, and this maps by inclusion to the class represented by $\xi$ in $H_i(M)$. The Poincar\'e dual in $H^{n-i}(M)$ of this class in $H_i(M)$ is then the image of $\D^{-1}([\xi])\in H^{n-i}(M,M-|\xi|)$ under the inclusion-induced $H^{n-i}(M,M-|\xi|)\to H^{n-i}(M)$. 

Beyond being a pleasant verification of consistency, this observation can be used fairly readily to see that the homology product induced by the Goresky-MacPherson intersection product of chains really is Poincar\'e dual to the usual cup product on compact manifolds. In fact, suppose $\xi\in C_i(M)$ and $\eta\in C_j(M)$ are two \emph{cycles} in  general position in the compact oriented PL $n$-manifold $M$. In this case, ``stratified'' general position reduces to the single general position requirement that $\dim(|\xi|\cap |\eta|)\leq i+j-n$. If $[\xi]\in H_i(|\xi|)$ and $[\eta]\in H_j([\eta])$ are the homology classes corresponding to the chains by Lemma \ref{L: useful}, then the definition of the Goresky-MacPherson intersection product in \cite[Section 2.1]{GM1}  reduces in this case to the chain corresponding to the image of $[\xi]\otimes [\eta]$ under the composition of maps
\begin{align}
H_i(|\xi|)\otimes H_j(|\eta|)&\xr{(\D\otimes \D)^{-1}} H^{n-i}(M, M-|\xi|)\otimes H^{n-j}(M, M-|\eta|)\notag\\
&\xr{\smile}H^{2n-i-j}(M,M-(|\xi|\cap|\eta|))\label{E: manifold pf}\\
&\xr{\D} H_{i+j-n}(|\xi|\cap|\eta|),\notag
\end{align}
though replacing the Goresky-MacPherson duality map with our corrected map $\D$. By Proposition \ref{P: compact}, the product in this case agrees with that previously studied in \cite{GBF18}.

We denote this intersection product by $\xi\pf \eta$.  It is shown in \cite[Theorem 1]{GM1} that this intersection product  leads to a well-defined product on homology $H_i(M)\otimes H_j(M)\to H_{i+j-n}(M)$, using that any two PL cycles are homologous to a pair of cycles in general position and that the homology class of the image depends only on the homology classes of the inputs. In fact, more generally, such products are defined for intersection homology groups of appropriate perversities on PL stratified pseudomanifolds.

We show that the homology intersection product on manifolds is dual to the cup product. By the Goresky-MacPherson result just cited, it is sufficient to verify this utilizing two representative cycles in general position.

\begin{theorem}\label{T: duals}
Let $M$ be a compact oriented PL $n$-manifold. Let $\xi\in C_i(M)$ and $\eta\in C_j(M)$ be PL cycles in general position and represented by classes $[\xi]\in H_i(|\xi|)$ and $[\eta]\in H_j([\eta])$. Then if $\pf$ is the Goresky-MacPherson intersection product we have a commutative diagram
\begin{diagram}[LaTeXeqno]\label{D: dual diagram}
H_i(|\xi|)\otimes H_j(|\eta|)&\rTo^{\mf i_{|\xi|}\otimes \mf i_{|\eta|}}&H_i(M)\otimes H_j(M)\\
&&\dTo^{(\D\otimes \D)^{-1}}\\
&&H^{n-i}(M)\otimes H^{n-j}(M)\\
\dTo^{\pf}&&\dTo^\smile\\
&&H^{2n-i-j}(M)\\
&&\dTo^\D\\
H_{i+j-n}(|\xi|\cap|\eta|)&\rTo^{\mf i_{|\xi|\cap|\eta|}}&H_{i+j-n}(M).
\end{diagram}
\end{theorem}

\begin{corollary}\label{C: the point}
Let $M$ be a compact oriented PL manifold of dimension $n$ and $\pf$ the Goresky-MacPherson intersection product. The following diagram commutes:
\begin{diagram}
H^{n-i}(M)\otimes H^{n-j}(M)&\rTo^\smile&H^{2n-i-j}(M)\\
\dTo^{\D\otimes \D}&&\dTo^{\D}\\
H_i(M)\otimes H_{j}(M)&\rTo^\pf&H_{i+j-n}(M).
\end{diagram}
\end{corollary}

\begin{proof}[Proof of Theorem \ref{T: duals}]
 Applying the definition of $\pf$ and Corollary \ref{C: duals}, we have 

\begin{align*}
\mf i_{|\xi|\cap |\eta|}[\xi\pf\eta]&=[\mf i_{|\xi|\cap |\eta|}\circ \D\circ \smile\circ (\D\otimes \D)^{-1} ]([\xi]\otimes[\eta])\\
&= [\D\circ \mf i^{|\xi|\cap |\eta|}\circ \smile\circ (\D\otimes \D)^{-1} ]([\xi]\otimes[\eta])\\
&=(-1)^n [\D\circ \mf i^{|\xi|\cap |\eta|}\circ \smile\circ (\D^{-1}\otimes \D^{-1})] ([\xi]\otimes[\eta])\\
&=(-1)^{n+ni}[\D\circ \mf i^{|\xi|\cap |\eta|}]( \D^{-1}([\xi])\smile  \D^{-1}([\eta]))\\
&=(-1)^{n+ni}\D ((\mf i^{|\xi|}\D^{-1}([\xi]))\smile (\mf i^{|\eta|}\D^{-1}([\eta])))\\
&=(-1)^{n+ni}\D (\D^{-1}(\mf i_{|\xi|}([\xi]))\smile  \D^{-1}(\mf i_{|\eta|}([\eta])))\\
&=(-1)^{n}[\D\circ\smile \circ (\D^{-1}\otimes \D^{-1}) \circ (\mf i_{|\xi|}\otimes \mf i_{|\eta|})]([\xi]\otimes[\eta])\\
&=[\D\circ\smile \circ (\D\otimes \D)^{-1} \circ (\mf i_{|\xi|}\otimes \mf i_{|\eta|})]([\xi]\otimes[\eta]).
\end{align*}

In the fifth line, we use naturality of the cup product coming from the maps $(M;\emptyset, \emptyset)\to (M; M-|\xi|,M-|\eta|)$. In particular,  $\mf i^{|\xi|\cap |\eta|}$ restricts on pairs to $\mf i^{|\xi|}$ and $\mf i^{|\eta|}$.  To explain the signs, we also recall that if $f,g$ are chain maps then the Koszul convention has $f\otimes g$ acting on an element $x\otimes y$ by $(f\otimes g)(x\otimes y)=(-1)^{|g||x|}f(x)\otimes g(y)$, where $|g|$ is the degree of $g$ as a chain map and $|x|$ is the degree of $x$ as a chain element. It also follows that 
\begin{align*}
\alpha\otimes \beta&=(\D^{-1}\D(\alpha))\otimes (\D^{-1}\D(\beta))\\
&=(-1)^{n(|\alpha|-n)}(\D^{-1}\otimes \D^{-1})((\D\alpha)\otimes (\D\beta))\\
&=(-1)^{n(|\alpha|-n)+n|\alpha|}(\D^{-1}\otimes \D^{-1})(\D\otimes \D)(\alpha\otimes \beta)\\
&=(-1)^n(\D^{-1}\otimes \D^{-1})(\D\otimes \D)(\alpha\otimes \beta),
\end{align*}
so $(\D\otimes \D)^{-1}=(-1)^n(\D^{-1}\otimes \D^{-1})$. 
\end{proof}

An alternative proof of Theorem 7.2 could likely be obtained from the naturality of the homology intersection product of Dold \cite[Section VIII.13]{Dold}. The composition down the right of Diagram \eqref{D: dual diagram} agrees with Dold's intersection product by \cite[Equation VIII.13.5]{Dold}, taking $(M,\emptyset)$ as both input pairs. By Proposition \ref{P: compact}, the Goresky-MacPherson intersection product on the left of Diagram \eqref{D: dual diagram} coincides with our product map $\mu$. 
The definition of $\mu$ is formally analogous to the definition of the homology intersection product in Dold, but they are not identical owing to the technicalities in the formulations of the various duality and transfer maps. These can probably be compared, but we will not pursue it here. 

Rather, we can turn around the observation that the composition down the right of Diagram \eqref{D: dual diagram} is Dold's homology intersection product to note that Corollary \ref{C: the point} demonstrates that the Goresky-MacPherson chain intersection product induces Dold's homology product. It then follows from Proposition \ref{P: compact} and Remark \ref{R: McC comp} that the homology products induced by our chain-level $\mu$ (equivalently McClure's $\mu_2$ of \cite{McC}) are the Dold intersection product up to signs. We state this as a further corollary:

\begin{corollary}\label{C: all agree}
Let $M$ be a compact oriented PL manifold. Then the following homology products $H_i(M)\otimes H_j(M)\to H_{i+j-n}(M)$ are equivalent (up to sign conventions):

\begin{enumerate}
\item the Goresky-MacPherson intersection product,
\item the homology product induced by our chain-level map $\mu$,

\item  the homology product induced by McClure's chain-level map $\mu_2$ \cite{McC},

\item  the Dold homology product \cite[Section VIII.13]{Dold}.
 \end{enumerate}
\end{corollary}

\subsection{Open questions: duality on pseudomanifolds without sheaves}\label{S: open}

The next reasonable question to ask is whether versions of 
Corollary \ref{C: duals} and Theorem \ref{T: duals} hold for intersection homology on PL stratified pseudomanifolds, utilizing the intersection homology Poincar\'e duality isomorphism and intersection homology cup and cap products of \cite{GBF25, GBF35}. In particular, if $X$ is a compact oriented PL stratified pseudomanifold and if $\xi\in I^{\bar p}C_i(X)$ is a cycle, then we have a version of $\mf i_{|\xi|}$ 
that runs $H_i(|\xi|)\to I^{\bar p}H_i(X)$ and takes the homology class representing $\xi$ in $H_i(|\xi|)$ to the intersection homology class represented by $\xi$ in $I^{\bar p}H_i(X)$. Letting $D\bar p$ be the dual perversity to $\bar p$ (see \cite[Definition 3.1.7]{GBF35}), one could then first seek a diagram

\begin{diagram}[LaTeXeqno]\label{D: IC bad natural}
H^{n-i}(X,X-|\xi|)&\rTo^{\mf i^{|\xi|}} &I_{D\bar p}H^{n-i}(X)\\
\dTo^\D&&\dTo^\D\\
H_i(|\xi|)&\rTo^{\mf i_{|\xi|}}&I^{\bar p}H_i(X),
\end{diagram}
analogous to the diagram of Corollary \ref{C: duals} but with the map $\D$ on the right now being the duality isomorphism\footnote{Technically, we should either work with field coefficients or with locally-torsion free spaces for the duality isomorphism $I_{D\bar p}H^{n-i}(X)\cong I^{\bar p}H_i(X)$ and the intersection cohomology cup product to be defined; see \cite[Chapter 8]{GBF35}. The reader can choose either option or safely ignore the issue for the point of this discussion.}
 of \cite{GBF25, GBF35}; as for the manifold case above, this isomorphism is simply the signed (intersection homology) cap product with the fundamental class. Then one could attempt to proceed on through the argument of Theorem \ref{T: duals} using these $\mf i$ maps. 

But there are  problems. First of all, the map $\D$ on the left here is not well defined because the hypotheses for Proposition \ref{P: GM dual} are not met, as the subspace in the pair $(|\xi|,\emptyset)$ does not contain  $\Sigma_X$. It fact, it is easy to construct examples where $H_i(|\xi|)$ and $H^{n-i}(X,X-|\xi|)$ are not even abstractly isomorphic. For example, let $X=S^n\vee S^n$ with some orientation on the spheres, and let $\xi$ be a fundamental class. Then $|\xi|=X$ and $H_n(X)\cong \Z\oplus \Z$, but $H^0(X,X-X)=H^0(X)\cong \Z$. Of course this does not immediately rule out the commutativity of some diagram of the form \eqref{D: IC bad natural}, but it does mean that we cannot define an intersection product by the composition \eqref{E: manifold pf}.

Indeed, this is not how Goresky and MacPherson define their intersection pairing for intersection chains. Instead, simplifying the construction of \cite[Section 2.1]{GM1} to cycles and using our duality map, their intersection product is determined instead by  the composition

\begin{align*}
H_i(|\xi|)\otimes H_j(|\eta|)&\to H_i(|\xi|\cup \Sigma, \Sigma)\otimes H_j(|\eta|\cup \Sigma, \Sigma)\\
&\xr{(\D\otimes \D)^{-1}} H^{n-i}(X-\Sigma, X-(|\xi|\cup\Sigma))\otimes H^{n-j}(X-\Sigma, X-(|\eta|\cap \Sigma))\\
&\xr{\smile}H^{2n-i-j}(X-\Sigma,X-((|\xi|\cap|\eta|)\cup\Sigma))\\
&\xr{\D} H_{i+j-n}((|\xi|\cap|\eta|)\cup\Sigma, \Sigma)\\
&\xleftarrow{\cong} H_{i+j-n}(|\xi|\cap|\eta|,|\xi|\cap|\eta|\cap \Sigma) \qquad\text{by excision.}
\end{align*}
With the assumptions on the perversities in \cite{GM1}, this last group is isomorphic to $H_{i+j-n}(|\xi|\cap|\eta|)$, though we also know from our work here (see Propositions \ref{P: G} and \ref{P: compact}) that an element of $H_{i+j-n}((|\xi|\cap|\eta|)\cup\Sigma, \Sigma)$ itself determines an intersection chain if $\xi$ and $\eta$ are intersection chains in stratified general position. Thus it seems that our replacement for the diagram of Corollary \ref{C: duals} might instead be something like

\begin{diagram}[LaTeXeqno]\label{D: IC bad natural2}
H^{n-i}(X-\Sigma,X-(|\xi|\cup\Sigma))&\rTo^{\mf i^{|\xi|}} &I_{D\bar p}H^{n-i}(X)\notag\\
\dTo^\D_\cong\\
H_i(|\xi|\cup\Sigma, \Sigma)&&\dTo^\D\\
\uTo\\
H_i(|\xi|)&\rTo^{\mf i_{|\xi|}}&I^{\bar p}H_i(X).
\end{diagram}
As both $\D$ maps are now isomorphisms again, we can still talk sensibly about commutativity of this diagram.
And, in fact, if $\bar p(Z)\leq \codim(Z)-2$ for all singular strata $Z$, as is the case for all perversities  in \cite{GM1}, then the map $H_i(|\xi|) \to H_i(|\xi|\cup\Sigma, \Sigma)$ is also an isomorphism: It is equal to the composition $$H_i(|\xi|) \to H_i(|\xi|,|\xi|\cap \Sigma) \to H_i(|\xi|\cup\Sigma, \Sigma),$$
the second map being an excision isomorphism and the first being an isomorphism from the long exact sequence of the pair,  as the perversity condition ensures $H_i(|\xi|\cap \Sigma)=H_{i-1}(|\xi|\cap \Sigma)=0$.

But now there is a different  problem: what is the natural definition of the map 
$\mf i^{|\xi|}: H^{n-i}(X-\Sigma,X-(|\xi|\cup\Sigma))\to I_{D\bar p}H^{n-i}(X)$?
There is no obvious chain map  $I^{D\bar p}C_*(X)\to C_*(X-\Sigma,X-(|\xi|\cup\Sigma))$ that would induce this cohomology map. We could conceivably \emph{define} $\mf i^{|\xi|}$ via diagram \eqref{D: IC bad natural2} in the case where all vertical maps are isomorphisms, but that's not good enough for a version of Theorem \ref{T: duals} as it is not clear how to show that a collection of $\mf i$ maps defined this way is natural with respect to the cup product without already having some kind of compatibility between cup products and appropriate homology pairings.

So we are stuck. There does not seem to be in the spirit of Theorem \ref{T: duals} and Corollary \ref{C: the point}
a similar proof of the duality isomorphism for pseudomanifolds between the Goresky-MacPherson intersection product and the intersection cohomology cup product  of \cite{GBF25}.  
However, such an isomorphism \emph{is} shown in \cite{GBF30} using sheaf-theoretic techniques, including the sheafification of the noncompact chain-level intersection product we have just developed, justifying our work here.

\bibliographystyle{amsplain}
\bibliography{../../bib}

Several diagrams in this paper were typeset using the \TeX\, commutative
diagrams package by Paul Taylor.

\end{document}